\documentclass[12pt]{amsart}
\usepackage{epsfig,psfrag,amsmath,amssymb,amsthm,bbm}
\renewcommand{\epsilon}{\varepsilon}
\renewcommand{\phi}{\varphi}
\renewcommand{\kappa}{\varkappa}

\newcommand{\cal}{\mathcal}
\newcommand{\G}{\mathcal{G}}
\newcommand{\IE}{\mathbb{E}}

\newcommand{\I}{\mathbbm{1}}
\newcommand{\IN}{\mathbb{N}}
\newcommand{\IP}{\mathbb{P}}
\newcommand{\IR}{\mathbb{R}}
\newcommand{\IZ}{\mathbb{Z}}

\newcommand{\Zb}{\bar{\bar{Z}}}
\newcommand{\Qb}{\bar{\bar{Q}}}   

\newtheorem{theorem}{Theorem}
\newtheorem{lemma}[theorem]{Lemma}
\newtheorem{corollary}{Corollary}
\newtheorem{proposition}[theorem]{Proposition}
\newtheorem{remark}{Remark}
\numberwithin{equation}{section}
\numberwithin{remark}{section}
\numberwithin{theorem}{section}
\numberwithin{corollary}{section}

\begin{document}

\title[Crossing velocities for a random walk in a random
potential]{Crossing velocities for an annealed random walk in a random
  potential}

\author{Elena Kosygina and Thomas Mountford} 
\begin{abstract}
  We consider a random walk in an i.i.d.\ non-negative potential on the
  $d$-dimensional integer lattice. The walk starts at the origin and
  is conditioned to hit a remote location $y$ on the lattice. We
  prove that the expected time under the annealed path measure needed
  by the random walk to reach $y$  grows only linearly in the
  distance from $y$ to the origin. In dimension one we show the
  existence of the asymptotic positive speed. 
 \end{abstract}
\maketitle
\section{Introduction}
\subsection*{Model description and main results}
Let $V(z,\omega),\ z\in\IZ^d$, be i.i.d.\ random variables on a
probability space $(\Omega,{\cal F},\IP)$, which represent a random
potential on $\IZ^d$. We assume that
\begin{equation}
  \label{V}
  V(0,\omega)\in[0,\infty]\text{ a.s.,\
  }\ \IP(V(0,\omega)=0)<1,\ \text{ and }\ 
  \operatornamewithlimits{essinf}\limits_{\omega\in\Omega}\,V(0,\omega)=0.
\end{equation}
\begin{remark}{\em The last equality is not needed for any of our
    results and could have been simply replaced by the condition
    $\IP(V(0,\omega)=\infty)<1$. In the case when the potential $V$ is
    bounded away from zero, Theorem~\ref{d2} below becomes very simple
    (see Section 2.2 of \cite{Zy09}). The last assumption makes the
    situation much more delicate, and we would like to emphasize this
    from the beginning. A good example to have in mind is when
    $V(0)\in\{0,1,\infty\}$, $\IP(V(0)=0)>0$, and $\IP(V(0)=1)>0$.}
\end{remark}

Let $P^x$ be the measure on the space of nearest-neighbor paths on
$\IZ^d$, which corresponds to a simple symmetric random walk $(S_n)_{n
  \geq 0}$ that starts at $x\in\IZ^d$.  The expectation with respect
to $P^x$
will be denoted by $E^x$. Let us fix $y\in\mathbb{Z}^d$, $y\ne x$, and
set $\tau_y=\inf\{n\ge 0\,:\,S_n=y\}$. For $\omega\in\Omega$ such that
\[Z^{\omega,x}_y=E^x\left(\I_{\{\tau_y<\infty\}}
e^{-\sum_{n=0}^{\tau_y-1}V(S_n,\omega)}\right)> 0\] define the quenched
path measure by
\begin{align}
  \label{qpm}
  Q^{\omega,x}_y(A):&=(Z^{\omega,x}_y)^{-1}E^x\left(\I_{\{\tau_y<\infty\}}
  \I_A\,e^{-\sum_{n=0}^{\tau_y-1}V(S_n,\omega)}\right).
\end{align}
The annealed path measure $Q^x_y$ is then given by
\begin{align}
  \label{apm}
  Q^x_y(A):&=(Z^x_y)^{-1}\IE E^x\left(\I_{\{\tau_y<\infty\}}
    \I_A\,e^{-\sum_{n=0}^{\tau_y-1}V(S_n,\omega)}\right)\\&=(Z^x_y)^{-1}
  \IE(Q^{\omega,x}_y(A)Z^{\omega,x}_y;Z^{\omega,x}_y>0),\nonumber
\end{align}
where $Z^x_y=\IE Z^{\omega,x}_y$.  These measures
have a natural interpretation in terms of the ``killed random walk'',
which we recall in the next subsection.

In the continuous setting, namely, for Brownian motion in a Poissonian
potential, the above path measures were introduced and studied by
A.-S.\ Sznitman (see \cite{Sz98} and references therein), T.\ Povel
(\cite{Po97}), M.\ W\"uthrich (\cite{Wu98}). In the
context of random walks, various aspects of these measures were
addressed in, for example, \cite{Ze98}, \cite{Fl08}, \cite{Zy09},
\cite{IV10a}.

The rates of decay of the quenched and annealed partition functions,
\begin{align}
  \label{al}
  \alpha_V(h):&=-\lim_{r\to\infty}\frac1r\log Z^{\omega,0}_{[rh]}\
  \text{ ($\IP$-a.s.)
    \ and}\\
  \label{be}\beta_V(h):&=-\lim_{r\to\infty}\frac1r\log Z^0_{[rh]},\ h\in \IR^d,
\end{align}
known also as the quenched and annealed Lyapunov exponents
respectively, are well defined (non-random) norms on $\IR^d$ (see
\cite{Ze98}, \cite{Fl08}, and \cite{KMZ10}; for the existence of
$\alpha_V(\cdot)$ it is sufficient to assume that $\IE
V<\infty$). Moreover, by Jensen's inequality, $\beta_V(h)\le
\alpha_V(h)$ for all $h\in \IR^d$.

In this paper we consider a random walk under the annealed path
measures $Q_y^0$ and address the question whether it is ballistic in
the sense that the average time it takes the walk to hit $y$,
$E_{Q^0_y}\tau_y$, grows linearly in $\|y\|$ as $\|y\|\to\infty$. The same
question regarding the quenched path measures for Brownian motion in a
Poissonian potential was positively resolved in (\cite{Sz95}). 
Our main results are contained in the following two theorems.

\begin{theorem}
  \label{d2} Let $V(z,\omega),\ z\in\IZ^d$ be i.i.d.\ under $\IP$ and
  satisfy (\ref{V}). If $d=1$, assume, in addition, that
\begin{equation}
  \label{d=1}
  \IP(V(0,\omega)\in(0,\infty))>0.
\end{equation}
Then there exists a constant $C\in(0,\infty)$ such that
\[
\limsup_{\|y\| \to \infty} \frac{E_{Q_y^0 } (\tau_y )}{\|y\|}\le
C.
\]
\end{theorem}
\begin{remark}{\em The condition (\ref{d=1}) is necessary:  if
    $d=1$ and our potential can take only two values, $0$ and
    $\infty$, both with strictly positive probability, then it can be
    shown that under the annealed measure as $ y \rightarrow \infty, $
    the process $y^{-1}S_{[s y^2]}\I_{\{s y^2 < \tau_y\}},\ s \ge 0$,
    converges in law to a Brownian excursion from $0$ to $1$, killed
    upon arriving at $1$.  In particular, $E_{Q_y^0 } (\tau_y )/y$
    converges to infinity as $y\to\infty$. This example runs counter
    to the ``natural assumption'' that the larger the potential the
    faster the random walk will achieve its target.}
\end{remark}
Theorem \ref{d2} readily leads to the following bound (the proof is given
in Section~\ref{two}).
\begin{corollary}\label{der2}
  For every unit vector $s\in
  \IR^d$ \[\frac{d\beta_{\lambda+V}(s)}{d\lambda}\Big|_{\lambda=0+}\le
C.\]
\end{corollary}

\begin{remark}
  {\em The existence of the above derivative follows from the concavity
  of the function $ \lambda \mapsto \beta_{\lambda+V}(h)$
  (see \cite[Theorem A(b)]{Fl07}).}
\end{remark}
In one dimension we can say more.
\begin{theorem}
\label{1d}
Let $d=1$ and $V(z,\omega),\ z\in\IZ^d$, satisfy the assumptions of
Theorem~\ref{d2}. Then there exists a constant $v\in(0,1)$ such that
\begin{equation}
  \label{v}
 \lim_{y\to\infty}\frac{E_{Q_y^0 }( \tau_y) }{y} = \frac{1}{v}. 
\end{equation}
Moreover,
\begin{equation}
  \label{der}
  \frac{d\beta_{\lambda+V}(1)}{d\lambda}\Big|_{\lambda=0+}=\frac{1}{v}.
\end{equation}
\end{theorem}

\subsection*{``Killed random walk'' description of the model}
Consider the following Markov chain (``killed random walk'') on
$\mathbb{Z}^d\cup \dagger$, where $\dagger$ is an absorbing state.  If
the walk is at $z\in\mathbb{Z}^d$ then with probability $1-e^{-V(z)}$
it goes to $\dagger$ and otherwise goes to one of the $2d$
nearest-neighbor sites with equal probabilities. We denote by
$\check{P}^{\omega,x}$ the measure, corresponding to this Markov chain
starting from $x$ in a fixed environment $V(z,\omega)$,
$z\in\mathbb{Z}^d$. Averaging over the environments gives the averaged
measure, $\check{P}^x(\cdot):=\IE\check{P}^{\omega,x}(\cdot)$. Let us
record the following obvious relations:
  \begin{align}
    \nonumber
    &Q^{\omega,x}_y(\cdot)=\check{P}^{\omega,x}(\cdot\,|\,\tau_y<\infty),\quad
    Z^{\omega,x}_y=\check{P}^{\omega,x}(\tau_y<\infty);
    \\\nonumber
    &Q^x_y(\cdot)=\check{P}(\cdot\,|\,\tau_y<\infty),\quad
    Z^x_y=\check{P}(\tau_y<\infty);
    \\&Q^{\omega,x}_y(A|B)=\check{P}^{\omega,x}(A\,|\,B\cap\{\tau_y<\infty\}).
    \label{mark3} 
  \end{align}
  The last equality will allow us to use the Markov property of the
  ``killed random walk'' to do computations under
  $Q^{\omega,x}_y$. Throughout the paper, when the starting point of a
  random walk is $0$ we shall often drop the superscript indicating
  the starting point.

\subsection*{Motivation and open problems}
There are several connections that motivate our interest and make us
believe that ballisticity is an important issue.

Recently, several works (\cite{Fl08}, \cite{Zy09}, \cite{IV10a})
addressed the question about the equality of quenched and annealed
Lyapunov exponents for small perturbations of a constant potential in
dimensions four and higher.  In particular, it was shown that when
$d\ge 4$ then under mild conditions on the potential for every
$\lambda>0$ there is a $\gamma^*>0$ such that for all
$\gamma\in(0,\gamma^*)$
\begin{equation}
  \label{eq}
  \beta_{\lambda+\gamma V}(\cdot)\equiv
\alpha_{\lambda+\gamma V}(\cdot).
\end{equation}
Recall the already mentioned fact that for $\lambda>0$ the random walk
under $Q_y$ is ballistic. Paper \cite{IV10a}, Theorem A, proves a
stronger result under even weaker conditions but still under the
restriction that $\lambda>0$. It is certainly an interesting question
whether (\ref{eq}) and its stronger version can be extended up to
$\lambda=0$ and whether $\gamma^*$ is locally uniform in $\lambda$ on
$[0,\infty)$. Such an extension, which is important in its own right,
would also help to clarify the relationship between the quenched and
annealed large deviations rate functions for random walks in a random
potential. This is the next connection that we would like to briefly
discuss.

Random walks in a random potential are more often considered under the
condition that they survive up to (a large) time $n\in\IN$ (see, for
example, \cite{Si95}, \cite{AZ96}, \cite{Kh96} and references
therein). The corresponding quenched and annealed measures with the
starting point $0$ are
\[Q^{\omega}_n(\cdot):=\check{P}^{\omega}(\cdot\,|\,\tau_\dagger>n);\quad
Q_n(\cdot):=\check{P}(\cdot\,|\,\tau_\dagger>n).\] It is known
(\cite{Ze98},\cite{Fl07}) that random walks under each of these
measures satisfy a full large deviation principle and the large
deviations rate functions, $I(\cdot)$ and $J(\cdot)$ respectively, are
given by the relations
\begin{align*}
  I(h)&=\sup_{\lambda\ge 0}(\alpha_{\lambda+V}(h)-\lambda);\\
  J(h)&=\sup_{\lambda\ge 0}(\beta_{\lambda+V}(h)-\lambda).
\end{align*}
Corollary~\ref{der2} implies, in particular, that for small $\|h\|$ we
have $J(h)=\beta_V(h)$. A similar result holds for $I(h)$ if the right
derivative of $\alpha_{\lambda+V}(s)$ with respect to $\lambda$ at
$\lambda=0+$ is bounded uniformly in $s$, $\|s\|=1$ (see
\cite[Corollary~2.3]{Sz95}, for the quenched result in a continuous
setting). If (\ref{eq}) were shown to hold also for $\lambda=0$ then
we would immediately conclude that for $d\ge 4$ and sufficiently small
$\gamma$ the large deviations rate functions $I$ and $J$ coinside in
some neighborhood of the origin.

For further details, connections with polymer measures, and open
problems we refer to the review \cite{IV10b}. 

\begin{remark}
  {\em After this paper was submitted for publication, we learned
    about the concurrent and completely independent work
    \cite{IV11}. The authors consider $d\ge 2$ and employ a different
    method, which allows them not only to show that the walk under
    $Q_y$ is ballistic (our Theorem~\ref{d2} for $d\ge 2$) but also to
    obtain the corresponding law of large numbers and central limit
    theorem (see Theorem~C of \cite{IV11}). For dimension one,
    Theorem~\ref{d2} can seemingly be also derived from the large
    deviations approach of \cite{GdH92}.  We believe, nonetheless, that
    our treatment is more direct and less technical.}
\end{remark}

\subsection*{Organization of the paper}
In Section~\ref{two}, we prove Theorem~\ref{d2} but only for $d > 1$.
The argument given does not seem to be adaptable to one dimension.
However, since Theorem~\ref{1d} implies Theorem~\ref{d2} for $d=1$, we
just need to prove the former. This is done in Section~\ref{three}
modulo several technical results (Lemmas \ref{taux}, \ref{probexp},
\ref{conw}, and \ref{conX}). The crucial  result among these is
Lemma~\ref{probexp}. Its proof, as well as proofs of other listed
above lemmas, are given in Section~\ref{tr} after the key
exponential estimate Theorem \ref{mess} is established in
Section~\ref{four}. Several elementary auxiliary results are collected
in the Appendix.

\section{Proof of Theorem~\ref{d2} for dimension higher than one}
\label{two} 

The quenched case in a continuous setting was investigated in
\cite{Sz95}. The argument given there applies to the quenched discrete
random walk with minor modification. Though we deal with the annealed
case, the basic division of space into occupied and unoccupied cubes
(see below) and the exploitation of lattice animal bounds are lifted
from \cite{Sz95}.

Let $d>1$ and $y\in\IZ^d\setminus\{0\}$ be the ``target point''. For
$A\subset \IZ^d$ define $\tau(A)=\inf\{n\ge 0:\,S_n\in A\}$. Fix a
large even $L$ and for $q\in\IZ^d$ let
$B(q)=(Lq+[-L/2,L/2)^d)\cap\IZ^d$. The set of these cubes, $\{B(q),\
q\in\IZ^d\}$, forms a partition of $\IZ^d$. Choose some
$\kappa\in(0,1)$ so that $\IP(V(0)\ge \kappa)>0$.  Given an
environment $\omega\in \Omega$ and $A\subset \IZ^d$ we shall say that
\[A\ \text{is {\em occupied} if }\
\max_{A\setminus\{y\}}V(x,\omega)\ge \kappa\ \text{and {\em empty}
  otherwise}.\] Denote by $\mathcal{O}=\mathcal{O}(\omega)$ the union
of all occupied cubes in our partition and by $\mathcal{O}^c$ the
union of all empty cubes.

\subsection*{Step 1} We shall estimate the time spent by our random walk in
$\mathcal{O}$. This is not difficult, since from every point in
$\mathcal{O}$ there is a path of length at most $d(L-1)$ to a point
where the potential is at least $\kappa$. This observation essentially
provides the proof of the following lemma, which is very much analogous to Theorem 1.1 of \cite{Sz95}.
\begin{lemma}
\label{lem211}
There exists constant $C_1=C_1(L,\kappa)$ such that for all
$y\in\IZ^d\setminus\{0\}$
\[
E_{Q_y} \left(\sum_{n=0}^{\tau_y-1} \I_{\{S_n \in
    \mathcal{O}\}}\right) \le C_1 \|y\|.
\]
\end{lemma}

\begin{proof}
  We shall show that there is an $\epsilon>0$ and
  $n_0=n_0(\epsilon,\kappa)$ such that for all $n\ge
  n_0$,
  \begin{equation}
    \label{ka}
    Q_y\left(\frac{1}{dL\|y\|}\sum_{i=0}^{\tau_y-1}
    \I_{\{S_i\in\mathcal{O}\}}>n\right)\le
  2e^{-\epsilon\kappa (n\|y\|-2)/2}.
  \end{equation}
  This will immediately imply the statement of the lemma.

  Define the stopping times $\sigma_m,\ m \in\IN$, by \[\sigma_1 \ = \
  \inf \{n \ge 0:\, S_n \in \mathcal{O}\}, \quad \sigma_{m+1} \ = \
  \inf \{n \ge \sigma_m + dL:\, S_n \in \mathcal{O}\}.\] The
  probability that during the time interval $[\sigma_m,\sigma_m +dL)$
  a simple symmetric random walk hits a point with the potential at
  least $\kappa$ and does not hit $y$ is greater than $(2d)^{-dL}$. Using
  the Markov property of the killed random walk we get that for
  $\epsilon\in(0,1)$ and all $m\ge 2$
  \begin{equation}
    \label{sita}
  \IE\check{P}^{\omega,0}\left(\sigma_m <
      \tau_y<\infty\right) \leq e^{-(m-1) \epsilon \kappa} +
 P(Y < (m-1) \epsilon ),
  \end{equation}
  where $Y$ is a binomial random variable with parameters $(m-1)$ and
  $(2d)^{-dL}$.  We choose $ \epsilon \in(0,(2d)^{-Ld})$ sufficiently
  small to ensure that $P(Y< (m-1)\epsilon ) < e^{-(m-1)\epsilon }$
  for all $m$ large. By (\ref{be}) we know that there is
  $\beta_0\in(0,\infty)$ such that $Z_y\ge e^{-\beta_0\|y\|}$ for all
  $y\in\IZ^d\setminus\{0\}$.  Dividing (\ref{sita}) by $Z_y$ and using
  (\ref{mark3}) and the last inequality we get (recall that
  $\kappa\in(0,1)$)
\[
Q_y (\sigma_m < \tau_y) \ \le \ 2e^{-(m-1) \epsilon \kappa }e^{\beta_0 \|y\|}.
\]
This completes the proof, since the set in the left hand side of
(\ref{ka}) is contained in $\{\sigma_{n\|y\|}<\tau_y\}$.\qedhere
\end{proof}

\subsection*{Step 2} To get a bound on the time spent by the walk in empty
cubes we need some information about sizes of connected components
of $\mathcal{O}^c$ under $Q_y$ (considered as a measure on
environments).

Given an $\omega\in\Omega$, we shall say that $x_1$ and $x_2$, 
both in $\IZ^d$, are connected in $\mathcal{O}^c$ if there is a simple random
walk path from $x_1$ to $x_2$ entirely contained in $\mathcal{O}^c$.
This defines a partition of $\mathcal {O}^c$ into connected
components.  If we consider a site percolation on $\IZ^d$ where the
site $q$ is open if and only if $B(q)$ is empty then standard percolation
results (see e.g.\ \cite{Gr99}) imply that for sufficiently large
$L$ all connected components of $\mathcal{O}^c$ are finite
$\IP$-a.s.. Since $Q_y$ is absolutely continuous with respect to
$\IP$, the same conclusion is true for $Q_y$-a.e.\ $\omega$. From now
on we suppose that $L$ is sufficiently large so that the above holds.

We shall need the following notation.  Let $D(x)=B(q)$ if $x\in
B(q)\cap\mathcal{O}$ and let $D(x)$ be equal to the connected
component of $\mathcal{O}^c$ that contains $x$ if $x\in
\mathcal{O}^c$.  Set $|D(x)|=\#\{q\in\IZ^d:\, B(q)\subset
D(x)\}$. Notice that $|D(x)|=1$ for every $x\in\mathcal{O}$. For
$D(x)\subset \mathcal {O}^c$ define the outer boundary
$\mathrm{Ad\,}D(x)$ as the union of all cubes in $\mathcal{O}$ which
are adjacent to $D(x)$, i.e.\ \[\mathrm{Ad\,} D(x)=\{z\in
\mathcal{O}:\,\exists\, x_1\in D(x),\ \exists\,z_1\in D(z)\ \text{such
  that }\|x_1-z_1\|=1\}.\] When $x\in\mathcal{O}$ we set $\mathrm{Ad\,}
D(x)=\emptyset$. The usual internal boundary of $D(x)$ will be denoted by
$\partial D(x)$, i.e. \[\partial D(x)=\{z\in D(x):\,\exists
z_1\notin D(x)\ \text{such that }\|z-z_1\|=1\}.\]

Consider a sequence of stopping times $(\rho_i)_{i \geq 0} $ and
an increasing sequence of sets $(A_i)_{i \geq 0}$ given by
$\rho_0=-1,\ A_0 = \emptyset $ and for $i\in\IN$,
\begin{align*}
  \rho_i&= \inf \{n > \rho_{i-1}: S_n \notin A_{i-1}  \},\\
  A_i&=A_{i-1} \cup D(S_{\rho_i}) \cup \mathrm{Ad\,} D(S_{\rho_i}).
\end{align*}
Note that $A_{i-1} \cap D(S_{\rho_i})=\emptyset$.  Finally, we introduce the
``discovery'' filtration $(\G_i)_{i \geq 1} $ where $\G_i $ is the
sigma field generated by $(S_{n \wedge \rho_i} )_{n \geq 0} $ and
$(V(x))_{ x \in A_{i-1}}$. 

\begin{lemma} \label{estimate} There exist strictly positive constants $c_2$
  and $C_2$ not depending on $L$ so that for all $i \ge 1$,
  $N\in\IN$, and all sufficiently large $L$
  \begin{equation}
    \label{est}
    Q_y(|D(S_{\rho_i})| = N,\,D(S_{\rho_i})\subset
\mathcal{O}^c\, |\, \G_i, ) \leq C_2 e^{-c_2NL^d}
  \end{equation}
\end{lemma}
\begin{proof}
  Define $q_i$ by the relation $S_{\rho_i}\in B(q_i)$. We first note
  that there are less than $(3^d)^{2N}$ distinct connected sets in
  $\IZ^d$ of cardinality $N$ containing $q_i$ (see e.g.\
  p.\,1009 of \cite{Sz95}). For each such set $\mathfrak{A}_N$, $q_i\in
  \mathfrak{A}_N$, define $D_N=\cup_{q\in \mathfrak{A}_N}B(q)$. It is
  sufficient to show that there are strictly positive constants $c_3$ and $C_3$, not depending on 
  $L$ or $N$,
  such that for every $D_N$ with $|D_N|=N$ and containing $S_{\rho_i}$
  \[Q_y(D(S_{\rho_i}) =D_N,\,D_N\subset
\mathcal{O}^c\,|\,\G_i ) \leq C_3 e^{-c_3NL^d}.\]
  
From this point on we fix $D_N$ (and so $\mathfrak{A}_N$).  We suppose
that $y $ is not in $D_N$ and leave it to the reader to make the minor
modifications for the case when $y\in D_N$.
  
  Given $\G_i$ and $D_N$, denote by $\Omega_{i,N}$ all environments
  which agree with $(V(x))_{x\in A_{i-1}}$ and have $D_N$ as a
  connected component of $\mathcal{O}^c$ containing $S_{\rho_i}$. This
  means, in particular, that $\mathrm{Ad\,} D_N\subset
  \mathcal{O}$. We need to get an upper bound on
  $Q_y(\Omega_{i,N}\,|\,\G_i)$. We shall compare this probability with
  the probability of the following modified set of environments. Let
  $B'(q)=Lq+[-L/4,L/4)^d$ and $D_N'=\cup_{q\in \mathfrak{A}_N}B'(q)$,
  $q_i\in \mathfrak{A}_N$. Denote by $\Omega'_{i,N}$ all environments
  which can be obtained from those in $\Omega_{i,N}$ by changing the
  potential only on $D'_N$ so that each center cube $B'(q)$,
  $q\in\mathfrak{A}_N$, becomes occupied.  A suitable upper bound on
  \begin{equation}\label{comp}
    \frac{Q_y(\Omega_{i,N}\,|\,\G_i)}{Q_y(\Omega'_{i,N}\,|\,\G_i)}
    = \frac{\IE\left(
      \check{P}^{\omega,S_{\rho_i}}  
       (\tau_y<\infty)\I_{\Omega_{i,N}}
        \,\big|\,\G_i\right
      )}{\IE \big(\check{P}^{\omega,S_{\rho_i}}(
       \tau_y<\infty)\I_{\Omega'_{i,N}}\,\big|\,\G_i\big)}
  \end{equation}
  will complete the proof of this lemma.  Let
  \begin{equation}\label{M}
    M^\omega_i=\max_{x\in
      D_N}E^x\left(e^{-\sum_{n=1}^{\tau_y-1}V(S_n,\omega)}
        \I_{\{\tau(D_N)>\tau_y\}}\I_{\{\tau_y<\infty\}}\right). 
  \end{equation}
  The expression we maximize can be non-zero only at $x\in\partial
  D_N$ (or at $x$ neighboring $y$ if $y\in
  D_N$).  Note that $M^\omega_i $ does not depend on the values of $V$ in $D_N$. 

  To bound the numerator in (\ref{comp}) we first observe that
  replacing the potential by $0$ in $D_N$ can only increase
  the expectation. Let $\nu_0=0$ and
  $\nu_{i+1}=\inf\{n>\nu_i\,:\,S_n\in\partial D_N\}$. Then, given
  $\G_i$, for $\omega\in\Omega_{i,N}$ we have 
  \begin{multline*}
    \check{P}^{\omega,S_{\rho_i}} (\tau_y<\infty)
    \le\\
    \sum_{k=0}^\infty
    E^{S_{\rho_i}}\left(e^{-\sum_{n=0}^{\nu_k}V(S_n)}
      E^{S_{\nu_k}}
      \left(e^{-\sum_{n=1}^{\tau_y-1}V(S_n)}
        \I_{\{\tau(D_N)>\tau_y\}}\I_{\{\tau_y<\infty\}}\right)
      \I_{\{\nu_k<\tau_y\}}\right)\\\le \sum_{k=0}^\infty
    M^\omega_i E^{S_{\rho_i}}\left(
      e^{-\sum_{n=0}^{\nu_k}V(S_n)}
      \I_{\{\nu_k<\tau_y\}} \right)=M^\omega_i\sum_{k=0}^\infty
    \check{P}^{S_{\rho_i}}\left(
      \nu_k<\tau_y \right).
  \end{multline*}
  For any point in $\partial D_N$ which is not adjacent
  to $y$, we have a uniform strictly positive lower bound,
  $(2d)^{-dL}$, of hitting a site $z$ with $V(z) \ge\kappa$ before
  returning to $D_N$. But to return to $D_N$ from
  $z$ the walk has to survive. It follows easily from the Markov
  property that \[\check{P}^{S_{\rho_i}}\left( \nu_k<\tau_y\right)\leq
  (1-(2d)^{-dL} (1- e^{-\kappa}))^{k}. \]
  Recall that,
  given $\G_i$,
  $M^\omega_i$ does not depend on the values of the potential on
  $D_N$. Thus, for all sufficiently large $L$
  \begin{multline*}
    \IE\left( \check{P}^{\omega,S_{\rho_i}}
      (\tau_y<\infty)\I_{\Omega_{i,N}} \,\big|\,\G_i\right ) \le
    \frac{(2d)^{dL}} {(1- e^{-\kappa})}\,\IE \left(M^\omega_i
      \I_{\Omega_{i,N}}\,\big|\,\G_i\right )\\=\frac{(2d)^{dL}} {(1-
      e^{-\kappa})}\,\IE (M^\omega_i\,\big|\,\G_i)\,\IP(
    \Omega_{i,N}\,\big|\,\G_i).
\end{multline*}
Finally, we shall get a lower bound on the denominator.  Denote by
$x_0$ a point where the maximum in (\ref{M}) is attained. Observe that
between any two points in $D_N\setminus D'_N$ there is a path of
length at most $dL(N+1)$ within this set.  In particular, there is
such a path from $S_{\rho_i}$ to $x_0$.  Thus we have
\begin{equation*}
  \IE \big(\check{P}^{\omega,S_{\rho_i}}(
  \tau_y<\infty)\I_{\Omega'_{i,N}}\,\big|\,\G_i\big)\ge
  (2d)^{-dL(N+1)}e^{-\kappa dL(N+1)}\,\IE
  (M^\omega_i\,\big|\,\G_i)\,\IP( \Omega'_{i,N}\,\big|\,\G_i).
\end{equation*}
Since $M^\omega_i$ does not depend on the potential in $D_N$, we can
now conclude that
\[
\frac{Q_y(\Omega_{i,N}\,|\,\G_i)}{Q_y(\Omega'_{i,N}\,|\,\G_i)}\le
\frac{(2d)^{dL(N+2)}e^{\kappa dL(N+1)}} {(1- e^{-\kappa})}\,\frac{\IP(
  \Omega_{i,N}\,\big|\,\G_i)}{\IP( \Omega'_{i,N}\,\big|\,\G_i)}
\]
The ratio of probabilities is bounded above
by \[\frac{(\IP(V(x)<\kappa))^{N(L/2)^d}}
{(1-(\IP(V(x)<\kappa))^{(L/2)^d})^N}\le e^{-c_4NL^d}\] for all
sufficiently large $L$. The last two bounds imply that there are
positive $c_3$ and $C_3$ not dependent on $L$ such that
\[Q_y(D(S_{\rho_i}) =D_N\, |\,\G_i ) = Q_y(\Omega_{i,N}\,|\,\G_i)\le
C_3e^{-c_3NL^d}\] as claimed.\qedhere
\end{proof}

\subsection*{Step 3} We now wish to show that, given $\G_{i} $, the
expected amount of time spent inside the (unknown) new component
$D(S_{\rho_i})$, by the random walk before $\tau_y $ is
of order one.  As Lemma~\ref{estimate} gives
very strong bounds on the size of this new component, all we need
is a crude upper bound on this expectation in terms of the size of
$D(S_{\rho_i})$.

We use the following lemma, which is basically an $h$-process result
(see \cite{Do01} for a general
exposition).
\begin{lemma}
\label{hprocess}
Consider a domain $D \subset \IZ^d $ of cardinality $N$.  Suppose that
the environment is such that for every (internal) boundary point $b\in
\partial D$ there exists a path from $b$ to a point $z$ with $V(z) \geq
\kappa$ which is entirely in the complement of $D \setminus \{b\}$
and of length less than $dL$.  Then for some universal $C_4 =
C_4(\kappa,L)$ and all $x \in D$ such that $Z^{\omega,x}_y>0$
\begin{equation}
  \label{gen}
  E_{Q_y^{\omega,x}}\left(\sum_{n=0}^{\tau_y-1} 
    \I_{\{S_n \in D\}}\right) \
  \le C_4N^{2/d} \log \left( 1+\sup_{u,v \in D}
  \frac{Z^{\omega,u}_y}
  {Z^{\omega,v}_y} \right).
\end{equation}
\end{lemma}

\begin{proof}
  Again we suppose that $y\not\in D$ and leave the remaining case to
  the reader. Consider the stopping times $( \nu _k)_{k \geq 0}$ where
  $\nu_0 =0$ and
\[
\nu_{k+1} = \inf \{ n > \nu_k: S_n \in \partial D \}.
\]
We have as in Lemma~\ref{estimate} that for any initial $x \in D$ and
$k\in\IN$,
\begin{equation}
  \label{nu}
  E^x \left(e^{-\sum_{n=0}^{\nu_k+dL}
    V(S_n)}\I_{\{\nu_k+dL<\tau_y\}}\right) \le (1-(2d)^{-dL}(1-e^{-
  \kappa}))^k .
\end{equation}

If $D$ has cardinality $N$, then for some constants $c_5$ and $C_5$
(depending on $d$) and all $x\in D$ we have $P^x(\tau(D^c)>t)\le C_5
e^{-c_5tN^{-2/d}}$.  This follows, since by the local central limit
theorem (see e.g.\ \cite{Du05}), there exist universal and nontrivial
constants $c$ and $C$ so that $P^x(S_{
  CN^{2/d}} \in D) < c < 1$ uniformly over $x \in D$.  Therefore,
\begin{equation}
  \label{srw}
  P^x\left(\sum_{n=0}^{\nu_k+dL} \I_{\{S_n \in D\}} >
  C_6N^{2/d}k ,\,\nu_k+dL<\tau_y\right) \leq e^{-c_6k}
\end{equation}
for some universal $c_6,\,C_6 \in(0,\infty)$.  Let $T=T(x,k)$ be the stopping
time when the number of steps in $D$ numbers more than
$C_6N^{2/d}k$. Then for any $x$ and any $k$,
\begin{equation*}
  Q_y^{x, \omega } \left(\sum_{n=0}^{\tau_y-1} \I_{\{S_n \in D\}}>
    C_6N^{2/d}k \right) =\frac{E^x\left(e^{-\sum_{n=0}^{T-1}V(S_n)}\I_{\{T<\tau_y\}} 
    Z^{\omega,S_T}_y\right)}{Z^{
    \omega,x}_y}.
\end{equation*}
Partitioning the path space into the event
$\left\{\sum_{n=0}^{\nu_k+dL}\I_{\{S_n\in D\}}>C_6N^{2/d}k \right\}$ and
its complement and using (\ref{nu}) and (\ref{srw}) we get that the
last ratio is dominated by
\begin{equation}
  \label{lnl}
  \left( e^{-c_6k} + (1-(2d)^{-dL}(1- e^{-\kappa}))^k \right) \sup_{u,v \in D}
\frac{Z^{\omega,u}_y}{Z^{\omega,v}_y}.
\end{equation}
If we choose now
\[ k=\left[C_7 m\log \left(1+\sup_{u,v \in
      D}\frac{Z^{\omega,u}_y}{Z^{\omega,v}_y} \right)\right],\] then
it is easy to see that for sufficiently large $C_7$ the expression in
(\ref{lnl}) will be less than $e^{-c_7m}$ for some strictly positive
$c_7$ depending only on our choice of $C_7$. This immediately implies
(\ref{gen}).
\end{proof}

\subsection*{Step 4} Now we can estimate the time spent by the random
walk in $D(S_{\rho_i})\cap\mathcal{O}^c$. First, we notice that for
some $c_8,\,C_8\in(0,\infty)$ and all
$D(S_{\rho_i})\subset\mathcal{O}^c$ 
\[\sup_{u,v\in D(S_{\rho_i})}\frac{Z^{\omega,u}_y}
{Z^{\omega,v}_y}\le C_8e^{c_8\kappa|D(S_{\rho_i})|dL}\] uniformly over
all environments for which $Z^{\omega,S_{\rho_i}}_y>0$. The above
bound is obtained simply by forcing the walk that starts at $v$ first
to go to $u$. For this we can choose a path, which is entirely
contained in $D(S_{\rho_i})$ and has length less than
$|D(S_{\rho_i})|dL$.  Then (\ref{gen}) and (\ref{est}) give us that
there is $C_9=C_9(\kappa,L)$ such that
\begin{equation}
  \label{to}
  E_{Q_y} \left(
    \sum_{n=0}^{\tau_y-1} \I_{\{S_n \in D(S_{\rho_i})\cap
      \mathcal{O}^c\}}\,\big|\, \G_i\right) \leq
C_9.
\end{equation}

\subsection*{Step 5.} This step will complete the proof of
Theorem~\ref{d2}. We have
\begin{equation}
  \label{la}
  E_{Q_y} (\tau_y ) = E_{Q_y} \left(\sum_{n=0}^{\tau_y-1} \I_{\{S_n \in
    \mathcal{O}\}}\right)+ E_{Q_y} \left(\sum_{n=0}^{\tau_y-1} \I_{\{S_n \in
    \mathcal{O}^c\}}\right).
\end{equation}
Lemma~\ref{lem211} takes care of the first term in the right hand
side. We have a uniform bound on the expected time spent in every
connected component of $\mathcal{O}^c$ visited by the random walk
prior to $\tau_y$. The only question we have to answer is how many of
these components it visited. Notice that in the time interval
$[\rho_i,\rho_{i+1})$ the walk necessarily visits a new occupied
cube. We define stopping times $(\beta_i)_{i \geq 0}$ by $\beta _0 =
0, \ \beta_{i+1} = \inf \{n > \beta_i : S_n\in \mathcal{O}\setminus
D(S_{\beta_i})\}$. Then, clearly, $\rho_i \geq \beta_i $. Arguing as
in the proof of Lemma~\ref{estimate} we have for some
positive $c_9$ that $Q_y ( \beta_{m\|y\|} < \tau_y ) < e^{-c_9m\|y\|}$.
The last term of (\ref{la}) is equal to
\begin{multline*}
  E_{Q_y} \left(\sum_{i=0}^\infty \I_{\{\rho_i < \tau_y\}}
    \sum_{n=0}^{\tau_y-1} \I_{\{S_n \in \ D(S_{\rho_i})\cap
      \mathcal{O}^c\}}\right)\\\leq E_{Q_y} \left(\sum_{i=0}^\infty
    \I_{\{\rho_i < \tau_y\}}E_{Q_y}\left(\sum_{n=0}^{\tau_y-1} \I_{\{S_n \in \
        D(S_{\rho_i})\cap
      \mathcal{O}^c\}} \,\big|\, \G_i\right)\right) \ \leq \ C_9
  E_{Q_y} \left(\sum_{i=0}^\infty \I_{\{\rho_i < \tau_y\}} \right)\\\le C_9
  \sum_{i=0}^\infty Q_y(\beta_i < \tau_y) \le C_{10}\|y\|.
\end{multline*}
for some $C_{10}$. The proof is now complete.\qed

In the remainder of the section we show how to derive Corollary
\ref{der2} from Theorem~\ref{d2}.  The following simple lemma
holds in all dimensions. Its proof is very similar to the
proof of Corollary~2.3 in \cite{Sz95}.  
\begin{lemma}
  \label{au}
For every unit vector $s\in
  \IR^d$ \[\frac{d\beta_{\lambda+V}(s)}{d\lambda}\Big|_{\lambda=0+}\le
  \limsup_{\|y\| \to \infty} \frac{E_{Q_y^0 } (\tau_y )}{\|y\|}.\]
\end{lemma}
\begin{proof}
  Fix any unit vector $s\in\IR^d$ and let
  $y=[rs]$. Observe that
 \begin{align*}
   -\frac{d}{d\lambda}\left(\liminf_{r\to\infty}\frac{1}{r}\,\log
     E_{Q_y}\left(e^{-\lambda \tau_y}\right)\right)\bigg|_{\lambda=0+}&=
   \frac{d}{d\lambda}\beta_{\lambda+V}(s)
   \bigg|_{\lambda=0+},\quad\text{and}\\
   -\liminf_{r\to\infty}\frac{1}{r}\left(\frac{d}{d\lambda}\log
   E_{Q_y}\left(e^{-\lambda \tau_y}\right)\bigg|_{\lambda=0+}\right)&=
   \limsup_{r\to\infty}\frac1{r}\,\frac{E_{Q_y}\left(\tau_y e^{-\lambda
         \tau_y}\right)}{E_{Q_y}\left(e^{-\lambda
         \tau_y}\right)}\bigg|_{\lambda=0+}\\&=
   \limsup_{\|y\|\to\infty}\frac{E_{Q_y}\tau_y}{\|y\|}.
 \end{align*}
 The statement of the lemma is an easy direction of the above exchange
 of limits. Since $\beta_{\lambda+V}$ is an increasing concave
 function of $\lambda$ on $[0,\infty)$ (see
 \cite[Theorem~A(b)]{Fl07}), it is enough to show that for each
 $\lambda>0$
\[
\limsup_{\|y\|\to\infty}\frac{E_{Q_y}\tau_y}{\|y\|}\ge
\frac{\beta_{\lambda+V}(s)-\beta_V(s)}{\lambda}.
\]
Let $b(0,y,V):=-\log \mathbb{E}E^0e^{-\sum_{n=0}^{\tau_y-1}
  V(S_n)}$. Then $b(0,y,\lambda+V)$ is a concave increasing
function of $\lambda$ on $[0,\infty)$ and 
\[E_{Q_y}\tau_y=\frac{d}{d\lambda}\left(-\log
  E_{Q_y}\left(e^{-\lambda
      \tau_y}\right)\right)\Big|_{\lambda=0}=\lim_{\lambda\to
  0+}\frac{b(0,y,\lambda+V)-b(0,y,V)}{\lambda}. \] By concavity, for
each $\lambda>0$,
\begin{align*}
  \lim_{\|y\|\to\infty} \frac{E_{Q_y}\tau_y}{\|y\|}
  &\ge\lim_{\|y\|\to\infty}
  \frac{b(0,y,\lambda+V)-b(0,y,V)}{\|y\|\lambda}\\&=\frac1\lambda
  \lim_{\|y\|\to\infty}\frac{b(0,y,\lambda+V)-b(0,y,V)}{\|y\|}=
  \frac{\beta_{\lambda+V}(s)-\beta_V(s)}{\lambda}.\qedhere
\end{align*}
\end{proof}
\section{Asymptotic speed in one dimension}\label{three}

We start by introducing some notation. For $x\in\IZ$ define
$\tau_x^{(1)}:=\tau_x$, and for $m\in\IN$ set
  \begin{align}
    \tau_x^{(m+1)}:&=\inf\{n>\tau_x^{(m)}\,:\,S_n=x\};
    \label{hits}\\
    \ell_y(x):&=\#\{n\in\{0,1,\dots,\tau_y-1\}:\,S_n=x\}.\label{lt}
  \end{align}
  In addition to $Q^\omega_y$ and $Q_y$ we shall need the
  following measures and partition functions:
  \begin{align*}
    &Q^\omega_{0,r}(\cdot)=
    \check{P}^\omega(\cdot\,|\,\tau_r<\tau_0^{(2)},\tau_r<\infty),\quad
    Z^\omega_{0,r}=\check{P}^\omega(\tau_r<\tau_0^{(2)},\tau_r<\infty);\\
    &Q_{0,r}(\cdot)= \check{P}(\cdot\,|\,\tau_r<\tau_0^{(2)},\tau_r<\infty),\quad
    Z_{0,r}=\check{P}(\tau_r<\tau_0^{(2)},\tau_r<\infty);\\
    &\Qb_{0,r}(\cdot)=
    \check{P}(\cdot\,|\,\tau_r<\tau_0^{(2)},\tau_r<\infty,\ell_r(x)\ge
    2,\,x\in\{1,2,\dots,r-1\});\\
    &\Zb_{0,r}=\check{P}(\tau_r<\tau_0^{(2)},\tau_r<\infty,\ell_r(x)\ge
    2,\,x\in\{1,2,\dots,r-1\}).
  \end{align*}
  Denote by $X_y$ the smallest non-negative integer in $[0,y]$, which
  is visited by the walk at most once up to the time $\tau_y$, i.e.
\begin{equation}
  \label{X_y}
  X_y=\min\{x\in\{0,1,\dots,y\}:\ \ell_y(x)\le 1\}.
\end{equation}
We shall refer to $X_y$ and all points between $0$ and $y$ inclusively
that were visited at most once up to time $\tau_y$ as ``renewal
points''. We use ``at most once'' instead of ``exactly once'' just to
include $y$ in the set of renewal points. The main idea of our proof
is to obtain (\ref{v}) using renewal theory.

The main ingredient of the proof of (\ref{v}) is
the following proposition.
\begin{proposition}\label{ingr}
  There is a constant $v\in(0,1)$ such
  that \[\lim_{y\to\infty}\frac{E_{Q_{0,y}}(\tau_y)}{y}=\frac1{v}.\]
\end{proposition}
We now indicate how this implies
(\ref{v}). The formal argument will be given after the proof of
Proposition~\ref{ingr}.  We have
\begin{multline}\label{deco}
  \frac{E_{Q_y}(\tau_y)}{y}=\frac{E_{Q_y}(\tau_{X_y})}{y}+
  \frac{1}{y}\,E_{Q_y}\left[E_{Q_y}(\tau_y-\tau_{X_y}\,|\,X_y)\right]
  \\=\frac{E_{Q_y}(\tau_{X_y})}{y}+
  \frac{1}{y}\,E_{Q_y}\left[E_{Q_{0,y-X_y}}(\tau_{y-X_y}\,|\,X_y)\right].
\end{multline}
The following lemma, whose proof is postponed until Section~\ref{tr},
takes care of the first term in the right-hand side of (\ref{deco}).
\begin{lemma}
 \label{taux}$\displaystyle\lim_{y\to\infty}\dfrac{E_{Q_y}(\tau_{X_y})}{y}=0$.
\end{lemma}
\noindent The second term in (\ref{deco}) will converge to $1/v$ by
Proposition~\ref{ingr}, provided that we have sufficient control on
$X_y$.

\begin{proof}[Proof of Proposition~\ref{ingr}]
  The proof relies on two technical lemmas. We shall state them as
  needed and supply proofs in Section~\ref{tr}.

Notice that $0$ is the first renewal point under
$Q_{0,y}$. Decomposition of the path space over all possible renewal
points in $[0,y]$ gives
\begin{align}\nonumber
  E_{Q_{0,y}}(\tau_y)&=\frac{E_{\check{P}}(\tau_y;\tau_y<\tau_0^{(2)},\tau_y<\infty)}
  {\check{P}(\tau_y<\tau_0^{(2)},\tau_y<\infty)}\\&=\frac{\sum\limits_{k=1}^y\,
    \sum\limits_{0=x_0<x_1<\dots<x_k=y}
    \prod\limits_{j=1}^k\Zb_{0,x_j-x_{j-1}}\sum\limits_{i=1}^k
    E_{\Qb_{0,x_i-x_{i-1}}}(\tau_{x_i-x_{i-1}})}
  {\sum\limits_{k=1}^y\sum\limits_{0=x_0<x_1<\dots<x_k=y}
    \prod\limits_{j=1}^k\Zb_{0, x_j-x_{j-1}}}.\label{decoz}
\end{align}
If the weights $\Zb_{0,r},\ r\ge 1$, formed a probability distribution
then the denominator would simply be the probability that $y$ is a
renewal point of a renewal sequence with this distribution, and the
numerator would be the expectation of some function of the renewal
lengths up to $y$ restricted to the set where $y$ is a renewal point.
But, as it turns out, these weights do not add up to 1. We shall have
to make an adjustment that is based on the following important fact.
\begin{lemma}\label{probexp}
  Let $\beta:=\beta_V(1)$ and $q(r):=e^{\beta r}\Zb_{0,r}$ Then there
  is an $\epsilon>0$ and $r_0>0$ such that for all $r\ge r_0$
\begin{equation}
  \label{massgap}
  q(r)\le e^{-\epsilon r}.
\end{equation}
Moreover, $\sum_{r=1}^\infty q(r)=1$.
\end{lemma}
\begin{remark}{\em
  A result analogous to $\sum_{r=1}^\infty q(r)=1$ is essentially
  shown in \cite{Zy09}.}
\end{remark}
\noindent Assume the above lemma.
Multiplying and dividing (\ref{decoz}) by
$e^{\beta y}$ and writing $e^{\beta
  y}$ as $\prod_{j=1}^ke^{\beta(x_j-x_{j-1})}$ we get
\begin{align}\nonumber
  &E_{Q_{0,y}}(\tau_y)=\\&\frac{\sum\limits_{k=1}^y\,
    \sum\limits_{0=x_0<x_1<\dots<x_k=y}
    \prod\limits_{j=1}^ke^{\beta(x_j-x_{j-1})}
    \Zb_{0,x_j-x_{j-1}}\sum\limits_{i=1}^k
    E_{\Qb_{0,x_i-x_{i-1}}}(\tau_{x_i-x_{i-1}})}
  {\sum\limits_{k=1}^y\sum\limits_{0=x_0<x_1<\dots<x_k=y}
    \prod\limits_{j=1}^ke^{\beta (x_j-x_{j-1})}\Zb_{0,
      x_j-x_{j-1}}}\nonumber\\&=\frac{\sum\limits_{k=1}^y
    \sum\limits_{0=x_0<x_1<\dots<x_k=y}
    \prod\limits_{j=1}^kq(x_j-x_{j-1})\sum\limits_{i=1}^k
    E_{\Qb_{0,x_i-x_{i-1}}}(\tau_{x_i-x_{i-1}})}
  {\sum\limits_{k=1}^y\sum\limits_{0=x_0<x_1<\dots<x_k=y}
    \prod\limits_{j=1}^kq(x_j-x_{j-1})} . \label{decoq} 
\end{align}
If we denote by $(X_i)_{i\ge 0}$ the renewal sequence with $X_0=0$
corresponding to the probability kernel $q(\cdot)$, by $A_y$ the event
that $y$ is a renewal point, and set $g(r)=E_{\Qb_{0,r}}(\tau_r)$,
then the above expression is equal to \[E_{Q_{0,y}}\Big(\sum_{i:\,X_i
  \le y}g(X_i-X_{i-1})\,\Big|\,A_y\Big). \] The next lemma provides a
bound on the growth of $g(r)$. This bound is not optimal but it will
be sufficient for our purposes as all we need to know is that $g(r)$
has subexponential growth.
\begin{lemma}\label{conw}
  There are constants $M_1$ and $M_2$ such that for all $r\ge
  1$ \[E_{\Qb_{0,r}}(\tau_r)\le M_1 r^3\ \text{and }\
  E_{Q_{0,r}}(\tau_r)\le M_2 r^3.\]
\end{lemma}
\begin{remark}{\em
  The last claim is not needed at this point. It will be used only in
  Section~\ref{tr}.}
\end{remark}
The law of large numbers and the renewal theorem tell us that, as
$y\to\infty$,
\begin{equation*}
  \frac{1}{y}\sum_{i:X_i \le
  y}g(X_i-X_{i-1})\overset{\text{$Q_{0,y}$-a.s.}}{\to}\frac {\sum_{r=1}^\infty
  g(r)q(r)}{\sum_{r=1}^\infty rq(r)}, \ \ \text{and } Q_{0,y}(A_y)\to
\frac{1}{\sum_{r=1}^\infty rq(r)}.
\end{equation*} 
The above relations together with Lemma~\ref{conw} and (\ref{massgap})
allow us to conclude that
\[\lim_{y\to\infty}\frac{E_{Q_{0,y}}(\tau_y)}{y}=
\frac{\sum_{r=1}^\infty g(r)q(r)}{\sum_{r=1}^\infty
  rq(r)}=:\frac1v<\infty.\qedhere\]
We note that the fact that $q(r) > 0 $ for $r >2$ and that for such $r, \ g(r) > r$ 
implies that $v$ is strictly less than $1$. 
\end{proof}

To proceed with the proof of (\ref{v}) we need one more auxiliary result.
\begin{lemma}
  \label{conX}
  There are $c,\,C\in(0,\infty)$ such that $Q_y(X_y>r)\le
  Ce^{-cr}$ for all $0\le r<y$.
\end{lemma}
\begin{proof}[Proof of (\ref{v}) in Theorem~\ref{1d}]
  By Lemma~\ref{taux} we only need to estimate the last term in
  (\ref{deco}). We fix an arbitrary $\epsilon>0$ and consider
  separately the expectation restricted to $\{X_y>\epsilon y\}$ and to
  its complement. By Lemmas~\ref{conw} and \ref{conX} we get
  \begin{multline*}
    E_{Q_y}\left[E_{Q_{0,y-X_y}}(\tau_{y-X_y}\,|\,X_y)\I_{\{X_y>\epsilon
        y\}}\right] \le M_2y^3Q_y(X_y>\epsilon y)\\\le CM_2y^3e^{-cy}\to
    0\ \text{ as } y\to \infty.
  \end{multline*}
  Consider now the expectation over $\{X_y \le \epsilon y\}$. We
  have
  \begin{equation}
    \label{b1}
    \frac{1}{y}\,
  E_{Q_y}\left[E_{Q_{0,y-X_y}}(\tau_{y-X_y}\,|\,X_y)\I_{\{X_y\le\epsilon
      y\}}\right] \le
  E_{Q_y}\left[\frac{E_{Q_{y-X_y}}(\tau_{y-X_y}\,|\,X_y)}{y-X_y}\,\I_{\{X_y\le
          \epsilon y\}}\right].
  \end{equation}
By Proposition~\ref{ingr}, for every
      $\epsilon_1>0$ there is $r_0$ such that for all $y\ge
      r_0/(1-\epsilon)$ and $x\le\epsilon
      y$
      \begin{equation}
        \label{cv}
        \left|\frac{Q_{0,y-x}(\tau_{y-x})}{y-x}-\frac{1}{v}\right| 
\le \epsilon_1.
      \end{equation}
      Thus, (\ref{b1}) is bounded above by $v^{-1}+\epsilon_1$ for all
      $y\ge r_0/(1-\epsilon)$.  On the other hand, by (\ref{cv}) and
      Lemma~\ref{conX} for all sufficiently large $y$
\begin{align*}
  \frac{1}{y}\,
  E_{Q_y}&\left[E_{Q_{0,y-X_y}}(\tau_{y-X_y}\,|\,X_y)\I_{\{X_y\le\epsilon
      y\}}\right] \ge \frac1y\left(\frac1v-\epsilon_1\right)E_{Q_y}\left
    ((y-X_y)\I_{\{X_y\le \epsilon
      y\}}\right)\\&\ge\left(\frac1v-\epsilon_1\right)(1-
  \epsilon)\,Q_y(X_y\le \epsilon y)\ge
  \left(\frac1v-\epsilon_1\right)(1-\epsilon)^2. 
\end{align*}
Since $\epsilon$ and $\epsilon_1$ were arbitrary, this finishes the proof.
\end{proof}

We close this section with a proof of (\ref{der}).  (See
\cite{Ze00} for a related result for random walks in random
environment.)

\begin{proof}[Proof of (\ref{der})]
  Lemma~\ref{au} implies
  that \[\frac{d\beta_{\lambda+V}(1)}{d\lambda}\Big|_{\lambda=0+}\le
  \frac1{v}.\]
It remains to show the converse inequality.  We fix $\epsilon > 0 .  $
It will be enough to show that for $\lambda $ positive and
sufficiently small
 \[
 \frac{\beta_{\lambda+V}(1) - \beta_{V}(1)}{ \lambda } \ \geq \
 \frac{1- \epsilon}{v}.
 \]
 By (\ref{be}) it is therefore sufficient to show that for such
 $\lambda$ fixed and all large $y$
 \[
 \frac{1}{\lambda y}\left(- \log E\left( e^{-\sum_{r=0}^{\tau_y -1}(V(S_r)+
       \lambda)}\right) + \log E\left(e^{-\sum_{r=0}^{\tau_y
         -1}V(S_r)} \right) \right) \ \geq \ \frac{1-
   \epsilon}{v}.
 \]
By Lemma~\ref{0}, this reduces to proving that
\begin{multline*}
  -\frac{1}{\lambda y}\, \log E_{Q_{0,y}}(e^{-\lambda
    \tau_y})=\frac{1}{\lambda y}\left(- \log
    E\left(e^{-\sum_{r=0}^{\tau_y -1}(V(S_r)+ \lambda)}
      \I_{\{\tau_0^{(2)} > \tau_y\}}\right)\right.\\\left.+ \log E\left(
      e^{-\sum_{r=0}^{\tau_y -1}V(S_r)}\I_{\{\tau_0^{(2)} > \tau_y\}}
    \right)\right) \geq \ \frac{1- \epsilon}{v}.
\end{multline*}
Thus, we need to show that for $\lambda$ small and all sufficiently
large $y$
\begin{equation}\label{need}
  E_{Q_{0,y}}(e^{-\lambda \tau_y})\le \exp\left(-\lambda y(1-\epsilon)/v\right).
\end{equation}
Conditioning on the number and locations of renewal points we get
\begin{equation}\label{rp}
  E_{Q_{0,y}}(e^{- \lambda
    \tau_y})=\sum_{k=1}^y\sum_{0<x_1<\dots<x_k=y}
    \prod_{j=1}^k q(x_j-x_{j-1})E_{\Qb_{0,x_j-x_{j-1}}} (e^{-
      \lambda \tau_{x_j-x_{j-1}}} ).
\end{equation}
The dominated convergence theorem implies that for each $r\in\IN$,
$\epsilon_1>0$, and all sufficiently small $\lambda$
 \[
 E_{\Qb_{0,r}} \left( \frac{1- e^{- \lambda \tau_r }}{ \lambda } \right) >
 (1- \epsilon_1) E_{\Qb_{0,r}}(\tau_r)
 \]
 and, thus,
 \begin{equation}
   \label{dc}
    E_{\Qb_{0,r}} ( e^{- \lambda \tau_r} ) < 1-\lambda (1- \epsilon_1)
 E_{\Qb_{0,r}}(\tau_r)\le  e^{-\lambda (1- \epsilon_1)
   E_{\Qb_{0,r}}(\tau_r) }.
 \end{equation}
 Next, we observe that for a renewal sequence based on the kernel
 $q(\cdot)$ and conditioned on having $y$ as a renewal point,
 $0<x_1<\dots<x_k=y$, there exists $M < \infty $ and $\epsilon_2 > 0$
 not depending on $y$, such that for all sufficiently large $y$ with
 probability at least $1- e^{-y \epsilon_2}$
 \begin{equation}\label{rm}
    \sum_{j=1}^k \I_{\{x_j - x_{j-1}\leq M\}}E_{\Qb_{0,x_j - x_{j-1}}}
    ( \tau_{x_j -x_{j-1}}) > \frac{y (1- \epsilon_1)}{v}.
\end{equation}
This statement follows from Lemma~\ref{probexp} and large deviations
bounds on i.i.d.\ random variables conditioned on an event of
probability bounded away from zero.

Now let us consider $\lambda >0$ sufficiently small and such that
(\ref{dc}) holds for each $r\in\{1,2,\dots, M\}$. Then (\ref{rp}) is
bounded above by
 \begin{multline*}
   \sum_{k=1}^y\sum_{0<x_1<\dots<x_k=y}
    \prod_{j=1}^k q(x_j-x_{j-1})
  e^{- \lambda (1-\epsilon_1)\I_{\{x_j-x_{j-1}\le M\}} E_{\Qb_{0,x_j-x_{j-1}}} (
    \tau_{x_j-x_{j-1}})}\\\le\sum_{k=1}^y\sum_{0<x_1<\dots<x_k=y} 
    \prod_{j=1}^k q(x_j-x_{j-1})e^{-y
  \lambda (1- \epsilon_1)^2/v}+e ^{- \epsilon_2 y}.
 \end{multline*}
We conclude that
\[ - \log \left( E(e^{-\sum_{r=0}^{\tau_y
       -1}(V(S_r)+ \lambda)} \I_{\{\tau_0 > \tau_y\}})\right) + \log \left(
   E(e^{-\sum_{r=0}^{\tau_y -1}V(S_r)}\I_{\{\tau_0 > \tau_y\}})\right)
 \]
\[
\geq -\log \left( e^{- \epsilon_2 y} + e^{-y\lambda (1- \epsilon_1 )^2/v} \right).
 \]
 This gives the desired inequality for $\lambda $ small, and we are done.
\end{proof} 

\section{The key environment estimate}\label{four}
A simple but important observation is that $Q_y$ and $Q_{0,y}$ can be
considered as measures not only on paths but on the product space of
paths and environments. Theorem~\ref{mess} below provides key estimates
on the environment under these measures. It is crucial for proofs of
technical results that we used in Section~\ref{three}.

Letters $a,b,x,z,x_i\ (i\in\IN\cup\{0\})$ will always denote
integers. Due to (\ref{d=1}) there exist $\kappa,K\in(0,\infty)$ such that
$\IP(V(0,\omega)\in[\kappa,K])>0$. Given an environment $\omega$, a
site $x\in\IZ$ will be called ``reasonable'' if
$V(x,\omega)\in[\kappa,K]$.

Let $I\subset[a,b]$ be an interval and $x_i\in I$, $i=1,2,\dots,m$,
$x_1<x_2<\dots<x_m$.  Define an ``environment event''
\begin{align*}
  \Omega_I(x_1,x_2,\dots,x_m)=
  \{\omega\in\Omega\,|\,V(x_i,\omega)&\in[\kappa,K]\ \forall
  i\in\{1,2,\dots,m\}\text{
    and }\\
  V(x,\omega)&\not\in[\kappa,K]\ \forall x\in
  I\setminus\{x_1,x_2,\dots,x_m\}\}.
\end{align*}
Observe that $\Omega_{(a,b)}$ is just the
event that $V(x,\omega)$ is not reasonable for every site 
$x\in(a,b)$.  We shall also need measures \[Q_{x,y}
(\cdot):=\check{P}^x(\cdot\,|\,\tau_x^{(2)}>\tau_y,
\tau_y<\infty),\quad 0\le x<y.\]

\begin{theorem} \label{mess} There exist constants $M_1,\ M_2$ and
  $\theta\in (0,1)$ not depending on $a,\ b,\ y$, or $x_i$,
  $i=1,2,\dots,m$, $0\le a=x_0<x_1<\dots<x_m<x_{m+1}=b\le y$, so that
  \begin{align}
    Q_{0,y}(\Omega_{(a,b)}(x_1, x_2,\cdots, x_m))\le \prod_{j=0} ^m
    (M_1 \theta ^{x_{j+1} -
      x_j}) \label{mess1};\\Q_y(\Omega_{(a,b)}(x_1, x_2,\cdots,
    x_m))\le \prod_{j=0} ^m (M_2 \theta ^{x_{j+1} -
      x_j}) \label{mess2}.
  \end{align}
\end{theorem}
This theorem is a consequence of Theorem~\ref{env} and Lemma~\ref{ind} below.
\begin{theorem}\label{env}
  There exist constants $M_3,\ M_4\ge 1$ and $\theta\in (0,1)$ not
  depending on $x,\ a,\ b$, and $y$, $0\le x\le a<b\le y$, such that
  \begin{align}
  Q_{x,y}(\Omega_{(a,b)})&\le M_3\theta^{b-a};\label{env1}\\
  Q_y(\Omega_{(a,b)})&\le M_4\theta^{b-a}.\label{env2}
  \end{align}
\end{theorem}

\begin{lemma}
\label{ind}
Let $0\le x<y$ and \[\cal{E}_{x-}\in \sigma(\{V(z,\omega)\},\
z<x),\ \ \cal{E}_{x+}\in \sigma(\{V(z,\omega)\},\ z>x).\] Then
\begin{align}
  Q_{0,y}(\cal{E}_{x-}\cap\{V(x)\in[\kappa,K]\}\cap\cal{E}_{x+}) &\le\frac{1} {1-e^{-\kappa}}\,Q_{0,x}(\cal{E}_{x-})Q_{x,y}(\cal{E}_{x+});\label{ind1}\\
  Q_y(\cal{E}_{x-}\cap\{V(x)\in[\kappa,K]\}\cap\cal{E}_{x+})
  &\le\frac{1}
  {1-e^{-\kappa}}\,Q_x(\cal{E}_{x-})Q_{x,y}(\cal{E}_{x+}).\label{ind2}
\end{align}
\end{lemma}
Lemma~\ref{ind} follows easily from the decomposition of the path
space according to the number of visits to a reasonable site $x$. The
details are given in the Appendix. The proof of Theorem~\ref{env} is
the main content of this section. For now we assume both statements and
show how they imply Theorem~\ref{mess}.
\begin{proof}[Proof of Theorem~\ref{mess}]
The proofs of (\ref{mess1}) and (\ref{mess2}) are identical, and we show
only (\ref{mess2}).
\begin{align*}
  Q_y(\Omega_{(a,b)}&(x_1,x_2,\dots,x_m))\\&=Q_y
  (\Omega_{(a,x_m)}(x_1,\dots,x_{m-1})
  \cap \{V(x_m)\in[\kappa,K]\}\cap \Omega_{(x_m,b)})\\
  &\overset{(\ref{ind2})}{\le}
  \frac{1}{1-e^{-\kappa}}\,Q_{x_m}(\Omega_{(a,x_m)}(x_1,\dots,x_{m-1})) 
  Q_{(x_m,y)}(\Omega_{(x_m,b)})\\
  &\overset{(\ref{ind1})}{\le} \left(\frac{1}{1-e^{-\kappa}}\right)^m
  Q_{x_1}(\Omega_{(a,x_1)})\prod_{j=1}^m Q_{x_j,x_{j+1}}
  (\Omega_{(x_j,x_{j+1})})\\&\overset{(\ref{env2})}{\le}
  \prod_{j=0}^m\left(\frac{M_3}{1-e^{-\kappa}}\,\theta^{x_{j+1}-x_j}\right).\qedhere
\end{align*}
\end{proof}
We turn now to the proof of Theorem~\ref{env}. It is a consequence of
several simple lemmas. We derive only (\ref{env2}), the proof of
(\ref{env1}) being practically the same.
\begin{lemma}
  \label{below1}
  Let $0\le a=x_0<x_1<\dots<x_m=b\le y$ and
  $\delta=\IP(V(0)\in[\kappa,K])\in(0,1)$. Then
  \begin{equation}
    \label{bel1}
    \frac{Q_b(\Omega_{(a,b)}(x_1,x_2,\dots,x_{m-1}))}{Q_b(\Omega_{(a,b)})}\ge 
    \frac{1}{2^m}\left(\frac{e^{-K}\delta}{1-\delta}\right)^{m-1}\prod_{j=1}^m
    \frac{1}{x_j-x_{j-1}}.
\end{equation}
\end{lemma}
We postpone the proof of this lemma to record its immediate corollary.
\begin{corollary}
\label{above1}
Set
\begin{equation*}
  C(a,b,\delta,K):=e^K\,\frac{1-\delta}{\delta}\sum_{m=1}^\infty  
  \sum_{a=x_0<x_1<\dots<x_m=b}\prod_{j=1}^m 
  \left(\frac{e^{-K}\delta}{2(1-\delta)}\right)
  \frac{1}{x_j-x_{j-1}}.
\end{equation*}
Then $Q_b(\Omega_{(a,b)})\le(C(a,b,\delta,K))^{-1}$.
\end{corollary}
The next lemma from renewal theory shows that the above inequality
actually gives an exponential bound on $Q_b(\Omega_{(a,b)})$.
\begin{lemma}
    \label{theta}
    Choose $\theta\in(0,1)$ so that
    \[\frac{e^{-K}\delta}{2(1-\delta)}\sum_{k=1}^\infty\frac{\theta^k}{k}=1.\] 
    Then there is a constant $c>0$ such that for all $a \le b, \
    C(a,b,\delta,K)\ge c\,\theta^{-(b-a)}$.
  \end{lemma}
  \begin{proof}
    Let $(\xi_n)_{n\ge 1}$ be random variables such that
    \[f_k=P(\xi_1=k)=\frac{e^{-K}\delta}{2(1-\delta)}
    \frac{\theta^k}{k},\quad k\in\IN,\] and define the renewal times
    $T_0=0$, $T_m=\sum_{j=1}^m\xi_j$, $m\in\IN$. Denote by $u_n$ the
    probability that $n$ is a renewal time. Then $u_0=1$
    and \[u_n=\sum_{k=1}^nf_ku_{n-k}>0,\quad n\ge 1.\] On the other
    hand,
    \begin{multline*}
      \sum_{m=1}^\infty\sum_{a=x_0<\dots<x_m=b}
    \prod_{j=1}^m\left(\frac{e^{-K}\delta}{2(1-\delta)}\right)
    \frac{1}{x_j-x_{j-1}}\\=
    \theta^{-(b-a)}\sum_{m=1}^\infty\sum_{a=x_0<\dots<x_m=b}
    \prod_{j=1}^m\left(\frac{e^{-K}\delta}{2(1-\delta)}\right)
    \frac{\theta^{x_j-x_{j-1}}}{x_j-x_{j-1}}=\theta^{-(b-a)}u_{b-a}.
    \end{multline*}
    By the renewal theorem (\cite{Fe68}, Ch.\,13, Sec.\,11),
    $u_n\to \mu^{-1}$ as $n\to\infty$, where $\mu=\sum_{n=1}^\infty
    nf_n<\infty$. This implies that $\min_{n\in\mathbb{N}}u_n>0$, and
    the claim follows.
  \end{proof}
\begin{proof}[Proof of Lemma~\ref{below1}]
  Denote the right hand side of (\ref{bel1}) by $C_m(\bar{x})$. Let
  $U$ be any potential on $(a,b)\setminus\{x_1,x_2,\dots,x_{m-1}\}$
  such that $U(x)\not\in[\kappa,K]$. Then it is enough to show that
  conditional on $V=U$ on $(a,b)\setminus\{x_1,x_2,\dots,x_{m-1}\}$
  \[Q_b(\Omega_{(a,b)}(x_1,x_2,\dots,x_{m-1})\,|\,U)\ge
  C_m(\bar{x})\,Q_b(\Omega_{(a,b)}|\,U).\] This is equivalent to the
  inequality
  \begin{equation}
    \label{bel2}
    \frac{\IE E^0\left(e^{-\sum_{n=0}^{\tau_b-1}V(S_n)};
        \Omega_{(a,b)}(x_1,x_2,\dots,x_{m-1})
        \,\big|\,U\right)}{\IE E^0\left(e^{-\sum_{n=0}^{\tau_b-1}
          V(S_n)};\Omega_{(a,b)}\,\big|\,U\right)}\ge
        C_m(\bar{x}).
  \end{equation}
  From now on assume that $V(x)=U(x)$ for all
  $x\in(a,b)\setminus\{x_1,x_2,\dots,x_{m-1}\}$ and drop the
  conditioning from the notation.

  Restricting the random walk expectation to those paths which on their
  way to $b$ hit every $x_i$, $i\in\{0,1,\dots,m-1\}$, only once, we
  obtain a lower bound on the numerator of (\ref{bel2}):
  \begin{align*}
    &\IE E^0\left(e^{-\sum_{n=0}^{\tau_b-1}V(S_n)};
      \left(\cap_{i=1}^m\{\tau_{x_{i-1}}^{(2)}>\tau_{x_i}\}\right)\cap
      \Omega_{(a,b)}(x_1,x_2,\dots,x_{m-1})
    \right)=\\
    & \IE E^0\left(e^{-\sum_{n=0}^{\tau_a}V(S_n)}\right)
    \left(\IE\left(e^{-V(0)};V(0)\in[\kappa,K]\right) \right)^{m-1}
    \times \\ &\makebox[4cm]{\ }\prod_{i=1}^mE^{x_{i-1}}
    \bigg(e^{-\sum\limits_{n=\tau_{x_{i-1}}+1}^{\tau_{x_i}-1}U(S_n)};\tau_{x_{i-1}}^{(2)}>\tau_{x_i}
    \bigg)\ge\\&\IE
    E^0\left(e^{-\sum_{n=0}^{\tau_a}V(S_n)}\right)\left(\frac{\delta}{e^K}
    \right)^{m-1} \prod_{i=1}^mE^{x_{i-1}}
    \bigg(e^{-\sum\limits_{n=\tau_{x_{i-1}}+1}^{\tau_{x_i}-1}U(S_n)};\tau_{x_{i-1}}^{(2)}>\tau_{x_i}
    \bigg)=\end{align*}\begin{align*} 
&\IE
    E^0\left(e^{-\sum_{n=0}^{\tau_a}V(S_n)}\right)\left(\frac{\delta}{e^K}
    \right)^{m-1} \prod_{i=1}^m\frac{1}{2(x_i-x_{i-1})}\times\\
    &\makebox[5cm]{\ }\prod_{i=1}^mE^{x_{i-1}}
    \bigg(e^{-\sum\limits_{n=\tau_{x_{i-1}}+1}^{\tau_{x_i}-1}U(S_n)}\big|
    \,\tau_{x_{i-1}}^{(2)}>\tau_{x_i}\bigg).
  \end{align*}
  To estimate the denominator of (\ref{bel2}), we first define the
  following random times:
\begin{align*}
  \sigma_i&=\sup\{n\le \tau_b:S_n=x_{i-1}\},\ i\in\{1,2,\dots,m\};\\
  \rho_i&=\inf\{n>\sigma_i:S_n=x_i\},\ i\in\{1,2,\dots,m\}.
\end{align*}
The denominator of (\ref{bel2}) is clearly bounded above by
\begin{align*}
  &\IE E^0\left(e^{-\sum_{n=0}^{\tau_a}V(S_n)}\right)\IE E^a
  \bigg(e^{-\sum\limits_{i=1}^m\sum\limits
    _{n=\sigma_i+1}^{\rho_i-1}U(S_n)};\Omega_{(a,b)}\bigg)\overset{\text{by (\ref{u})}}{\le}\\
  &\IE
  E^0\left(e^{-\sum_{n=0}^{\tau_a}V(S_n)}\right)(1-\delta)^{m-1}\prod_{i=1}^m
  E^{x_{i-1}}
  \bigg(e^{-\sum\limits_{n=\tau_{x_{i-1}}+1}^{\tau_{x_i}-1}U(S_n)}\big|
  \,\tau_{x_{i-1}}^{(2)}>\tau_{x_i}\bigg).
\end{align*}
Dividing the lower bound on the numerator by the upper bound on the
denominator of (\ref{bel2}) we obtain the statement of the lemma.
\end{proof}
We summarize the results of Corollary~\ref{above1} and
Lemma~\ref{theta}.
\begin{corollary}\label{above2}
  There are constants $C>0$ and $\theta\in(0,1)$ such that for all
  $0\le a<b\le y$ \[Q_b(\Omega_{(a,b)})\le C\theta^{b-a}.\]
\end{corollary}
Corollary~\ref{above2} gives us (\ref{env2}) in
the case when $b=y$.  

\begin{proof}[Proof of Theorem~\ref{env}]
  We would like to get a bound on $Q_y(\Omega_{(a,b)})$. 
  Let $X=\min\{x\ge b:
  V(x)\in[\kappa,K]\}$. Since $Q_y$ (restricted to environment events)
  is absolutely continuous with respect to $\IP$,
  $Q_y(X=\infty)=0$. Using
  Lemma~\ref{ind} and Corollary~\ref{above2} we get
  \begin{align*}
    Q_y&(\Omega_{(a,b)})=\sum_{m=b}^\infty Q_y(\Omega_{(a,b)}\cap
    \{X=m\})\\&
    =\sum_{m=b}^{y-1}Q_y(\Omega_{(a,m)}\cap\{V(m)\in[\kappa,K]\})+\sum_{m=y}^\infty
    Q_y(\Omega_{(a,b)}\cap \{X=m\}) \\&\le
    \frac{1}{1-e^{-\kappa}}\,\sum_{m=b}^{y-1}Q_m(\Omega_{(a,m)})+\sum_{m=y}^\infty
    Q_y(\Omega_{(a,y)})\delta(1-\delta)^{m-y}
    \\&\le \frac{C\theta^{b-a}}{(1-e^{-\kappa})(1-\theta)}+C\theta^{y-a}\le
    M_3\theta^{b-a} 
  \end{align*}
for some constant $M_3$.
\end{proof}

\section{Proofs of technical lemmas}\label{tr}
Theorem~\ref{mess} gives us good control on environments under
probability measures $Q_y$ and $Q_{0,y}$ and we are now in a position
to prove Lemmas~\ref{conX}, \ref{probexp}, and \ref{taux}. 

We start with two auxiliary statements, Lemmas~\ref{expbd} and
\ref{geomb}. Recall that, given $\omega\in \Omega$, a site $x\in\IZ$
is called ``reasonable'' if $V(x,\omega)\in[\kappa,K]$, where
$0<\kappa<K<\infty$.  In Lemma~\ref{expbd} we argue that outside of an
event of exponentially small (in $y$) probability with respect to
$Q_y$ or $Q_{0,y}$ there are of order $y$ reasonable sites in
$(0,y)$. Lemma~\ref{geomb} shows that having so many reasonable sites
in $(0,y)$ and not having a renewal in $(0,y)$ is very unlikely. These
two facts easily imply Lemmas~\ref{conX} and
\ref{probexp}. Lemma~\ref{taux} requires additional steps, and its
proof takes the rest of the section.

Denote by $R_y$ the set of all reasonable sites in $\{0,1,\dots,y\}$
and by $|R_y|$ the number of elements in $R_y$.
\begin{lemma}
\label{expbd}
There exist $M_5$, $M_6$ and $\nu_1\in(0,1)$ so that for all $y>0$ and all
intervals $I\subset(0,y)$
\begin{align}
  &Q_{0,y}(|R_y\cap I|\le \nu_1|I|)\le M_5 e^{-\nu_1 |I|};\label{expbd1}\\
  &Q_y(|R_y\cap I|\le \nu_1|I|)\le M_6 e^{-\nu_1 |I|}.\label{expbd2}
\end{align}
\end{lemma}

\begin{proof}
  We shall prove only (\ref{expbd1}). Let $I=(a,b)$, $0\le a<b\le
  y$, and $r=b-a-1=|I|$. Notice that \[\{|R_y\cap I|\le \nu_1 r\} \ = \ 
  \bigcup_{k=0}^{[\nu_1 r]}\ \bigcup_{a<x_1<\dots<x_k<b}\Omega_{I}(x_1,
  x_2, \cdots x_k),\] where for $k=0$ the second union reduces to a
  set $\Omega_I$. Thus, by Theorem~\ref{mess}
\[
Q_{0,y}(|R_y\cap I|\le \nu_1|I|)\le \sum_{k=0}^{[\nu_1 r]}
\binom{r}{k} M_1^{k+1} \theta ^{r+1}\le 2 \binom{r}{[\nu_1 r]}
M_1^{r\nu_1+1} \theta ^{r+1}
\]
for $\nu_1$ small.  Applying Stirling's formula to this bound, we have
that for some universal $C$,
\[
Q_{0,y}(|R_y\cap I|\le \nu_1|I|) \le CM_1 \theta^r\left(
  \frac{1}{1-\nu_1} \right) ^{(1- \nu_1)r} \left( \frac{M_1}{\nu_1}
\right) ^{\nu_1 r} < C M_1\left( \frac{1+ \theta}{2} \right) ^r
\]
for all $r$ provided that $\nu_1$ is chosen sufficiently small to ensure that
 \[(1-\nu_1)|\log(1-\nu_1)|+\nu_1|\log \nu_1|+\nu_1\log
M_1<\log\frac{1+\theta}{2\theta}.\qedhere\]
\end{proof}

For $0\le a<b\le y$ we define $B(a,b,y)$ to be the event that there
are no renewal points in interval $(a,b)$ up to time $\tau_y$.  The
event $K(a,b,\ell)$ contains all environments with at least $\ell$
reasonable sites in interval $(a,b)$.
\begin{lemma} \label{geomb} There exist nontrivial constants $C_2,\nu_2
 $ so that uniformly over $\ell$ and $0 \le a < b \le y$,
\[
Q_y(K(a,b,\ell) \cap B(a,b,y))  \le C_2 e^{- \nu_2 \ell}.
\]
\end{lemma}
Our proof is based on a coupling with a simple asymmetric random
  walk. 
  We shall use the following two elementary lemmas.
\begin{lemma}
\label{Zzz}
Let $Y_n$, $n\ge 0$, be a simple asymmetric random walk with the
rightward probability $p\in(1/2,1)$. Let $B(0,m)$ be the event that
there are no renewal points in $(0,m)$.  There exists $c = c(p) >0$ so
that for each $m>1$,
\[
P(B(0,m)|Y_0=0) \leq e^{-cm}.
\]
\end{lemma}
This statement follows from much more general arguments in \cite{Sz00}
(Lemma 1.2) and \cite{Zy09} (Proposition 4.3). Since the proof in our
case is basic, we give it in the Appendix for completeness.

Denote the first $\ell$ reasonable sites in $(a,b)$ by $x_i$,
$i\in\{1,2,\dots,\ell\}$.  It is easily seen that we may suppose
without loss of generality that the sequence $x_i $ extends to
infinity in interval $[b, \infty)$.
\begin{lemma}
  \label{ret1}
  Let $r\in\IN$, $x_i,x_{i+1},\dots,x_{i+r}$ be reasonable points, and
  $x_i<x_{i+1}<\dots<x_{i+r}<y$. Then for all $\omega\in\{Z^{\omega,x_i}_y>0\}$
  \[(a)\ Q^{\omega,x_i}_y(\tau_{x_{i+r}}>\tau_{x_i}^{(2)})\le e^{-\kappa};\quad
   (b)\ Q^{\omega,x_{i+r}}_y(\tau_{x_i}<\tau_y)\le e^{-\kappa r}.\]
\end{lemma}
The proof is given in the Appendix.

\begin{proof}[Proof of Lemma~\ref{geomb}] Fix $r$ so that $e^{-\kappa
    r}\le 1/3$. We say that a reasonable site $x_i$ is {\it alive} if
  $\tau_{x_{i+r}}<\tau_{x_i}^{(2)}$.  By Lemma~\ref{ret1} and our
  choice of $r$, the probability that the random walk returns to an
  alive site prior to $\tau_y$ is less than $1/3$.  Of course, the
  events $\{x_i \mbox{ is alive}\}$ and $\{x_{j} \mbox{ is alive}\}$
  are correlated for $|i-j| < r$ but otherwise independent.  Thus, in
  order to secure some independence, we define $J \subset \{ x_1, x_2,
  \dots, x_\ell\}$ as $\{ x_{\lambda_1}, x_{\lambda _2}, \dots,
  x_{\lambda_m}\}$, where
\[\lambda_0 = 0, \ \lambda_{i+1} = \inf \{ j \ge
\ \lambda_i+ r:\, \tau_{x_{j+r}}<\tau_{x_j}^{(2)}\}.\]  

Points of $J$ are called {\it good} points.  They are our candidates
for renewal points.  Again we can extend good points infinitely
outside interval $(a,b)$.

Let $C(h,a,b) $ be the event that interval $(a,b)$ contains less than
$h$ good points.  Denote by $N$ a binomial random variable with
parameters $[\ell/r]$ and $1-e^{-\kappa}$. We first note that by the
strong Markov property, given that $ \omega $ is in $K(a,b,\ell)\cap
\{Z^{\omega,0}_y>0\}$, 
\[
Q^{0, \omega}_y\left(C\left(\frac{(1-e^{- \kappa
      })\ell}{2r},a,b\right)\right) \leq P\left(N <
  \frac{(1-e^{- \kappa })\ell}{2r} \right) < \ e^{-\epsilon_2 \ell},
\]
for some small universal $\epsilon _2$. 
We now suppose that event $C((1-e^{- \kappa })\ell/(2r),a,b)$ does
not occur, so that there are at least $(1-e^{- \kappa })\ell/(2r)$ good
points.

By Lemma~\ref{ret1} and our choice of $r$ we can couple our process
$S$ and a simple random walk, $Y$, with rightward probability
$p=2/3$ so that (even though $\tau_{x_{\lambda_j}} $ is not a
stopping time) for each $i > 0$,
\[
\big\{ j \leq i: x_{\lambda_j} \mbox{ is visited  in } (\tau^S _{x_{\lambda_i}},
\tau^S _{x_{\lambda_{i+1}}}) \mbox{ by }S \big\}
\]
is a subset of 
\[
\big\{ j \leq i: j \mbox{ is visited in } (\tau^Y _{i},
\tau^Y_{{i+1}}) \mbox{ by }Y \big\}.
\]
This coupling necessarily entails that for $1 \leq i \leq
(1-e^{- \kappa })\ell/(2r)$ \[\{i \mbox{ is a
  renewal point for }Y \}\subset \{x_{\lambda_ i}
\mbox{ is a renewal point for }S \}.\]
  
The result follows since $K(a,b,\ell) \cap B(a,b,y)$ is contained in
the union of events $C((1-e^{- \kappa
    })\ell/(2r),a,b)$ and $\{Y$ has no renewal points in
$(0,(1-e^{- \kappa })\ell/(2r))\}$.
\end{proof}

Now we can easily derive Lemma~\ref{conX}.
\begin{proof}[Proof of Lemma~\ref{conX}]
  Let $0\le r<y$. Then for some fixed $\nu_1\in(0,1)$ by
  Lemma~\ref{geomb} and (\ref{expbd2}) we get
  \begin{align*}
    Q_y(X_y>r)&=Q_y(X_y>r; K(0,r,\nu_1 r))+ Q_y(X_y>r;
    K^c(0,r,\nu_1 r))\\ &\le Q_y(K(0,r, \nu_1 r)\cap
    B(0,r,y))+Q_y(K^c(0,r,\nu_1 r))\\&\le C_2e^{-\nu_2 \nu_1
      r}+M_6 e^{-\nu_1 r}.\qedhere
  \end{align*}
\end{proof}

We are also ready to prove Lemma~\ref{probexp}.  
\begin{proof}[Proof of Lemma~\ref{probexp}]
 We have by (\ref{expbd1}) and
  Lemma~\ref{geomb} that there is $\nu_1\in(0,1)$ such that
  \begin{align*}
    \Zb_{0,r}&=\Qb_{0,r}(K^c(0,r,\nu_1
    r))\Zb_{0,r}+\Qb_{0,r}(K(0,r,\nu_1 r))\Zb_{0,r}\\&\le
    Q_{0,r}(K^c(0,r,\nu_1 r))Z_{0,r}+Q_r(K(0,r,\nu_1 r)\cap
    B(0,r,r))Z_r\\&\le (M_5e^{-\nu_1r}+C_2e^{-\nu_2\nu_1 r})Z_r.
  \end{align*}
  This proves (\ref{massgap}) as $(\log Z_r)/r\to-\beta$ as
  $r\to\infty$. 
 
  We also have the following fact, which we prove in the Appendix.
  \begin{lemma}
      \label{0}
$\displaystyle\lim_{y\to\infty}\dfrac 1y\log Z_{0,y}=-\beta$.
    \end{lemma}
    This result immedaitely implies that the power series
    $\sum_{y=1}^\infty s^y e^{\beta y}Z_{0,y}$ has radius of
    convergence equal to $1$.  But decomposing $Z_{0,y}$ according to
    its renewal points (see the denominator in (\ref{decoq})) gives
    \[
    \sum_{y=1}^\infty s^y e^{\beta y}Z_{0,y} \ = \ \sum_{k=1}^\infty \left( \sum_{r=1}^\infty q(r)s^r  \right) ^k
    \]
    from which the conclusion that $\sum_{r=1}^\infty q(r) \ = \ 1$ is immediate.
\end{proof}

\begin{proof}[Proof of Lemma~\ref{taux}]
We have
\begin{equation}\label{tf}
  \frac{1}{y}\,E_{Q_y}(\tau_{X_y})= 
  \frac{1}{y}\,E_{Q_y}\Big(\sum_{ x\le -y^{1/4}}\ell_y(x)\Big) + 
  \frac{1}{y}\,E_{Q_y}\Big(\sum_{-y^{1/4}<x\le X_y}\ell_y(x)\Big).
\end{equation}
We shall
need the following two inequalities. Their proofs can be found in the Appendix.
\begin{lemma} \label{lemA} Let
  $B\in \sigma (\{V(x, \omega),x < 0\})$. Then
\[Q_y (B) \le 2y \IP(B).\]
\end{lemma}

\begin{lemma}
  \label{bias} For every $z\le x\le y$, $m\in\IZ$, and
  $\IP$-a.e. $\omega\in \{Z^\omega_y>0\}$
      \[Q_y^\omega(\ell_x(z)>m)\le P^0(\ell_x(z)>m).\]
    \end{lemma}
Now we can estimate the first term in the right hand side of (\ref{tf}).
  \begin{lemma} 
\label{bd1}
$\displaystyle\lim_{y\to\infty}E_{Q_y} \Big(\sum_ {x \le -y^{1/4} }
\ell_y(x)\Big) =0$.
\end{lemma}
\begin{proof}
  First, we note that by Lemma~\ref{vn} it is enough to show that for
  $R$ fixed \[\lim_{y\to\infty} E_{Q_y}\Big(\sum_{n=0}^{\tau_y-1} \I_{\{-Ry
      \leq S_n \le -y^{1/4}\}}\Big)=0.\] We introduce the event
 \[A_{\delta,y} = \{\# \text{ of sites in $(-y^{1/4}, 0)$ that are reasonable
  is less than } \delta y^{1/4}/2 \}.\] We have immediately
  that
  \begin{multline}\label{ay}
    E_{Q_y}\Big(\sum_{n=0}^{\tau_y-1} \I_{\{-Ry \leq S_n \leq
    -y^{1/4}\}}\Big)\\=E_{Q_y}\Big(\sum_{n=0} ^{\tau_y-1} \I_{\{-Ry \leq S_n \leq
    -y^{1/4}\}}\I_{A_{\delta,y}}\Big)+E_{Q_y}\Big(\sum_{n=0}^{\tau_y-1} \I_{\{-Ry
    \leq S_n \leq -y^{1/4}\}}\I_{A_{\delta,y}^c}\Big).
\end{multline}
By Lemma~\ref{bias}, the first term is bounded by
\begin{multline*}
  Q_y(A_{\delta,y})
  E^0\Big(\sum_{n=0}^{\tau_y-1} \I_{\{-Ry \leq S_n \leq -y^{1/4}\}}\Big) \\
  \overset{\text{Lemma~\ref{lemA}}}{\le}2y
  \IP(A_{\delta,y})\sum_{-Ry\le x\le-y^{1/4}}E^0(\ell_y(x))\le Cye^{-c(\delta) y^{1/4}}R^2 y ^2
\end{multline*}
for a universal $C$. In the last line we used standard large
deviations bounds for Bernoulli random variables (see
e.g. \cite{DZ98}) and the fact that for every $-Ry < x < y$ the local
time $\ell_y(x)$ is stochastically dominated by a geometric random
variable of parameter $(2y(R+1))^{-1}$. Therefore, it remains to deal
with the last term in (\ref{ay}).  But by part (b) of Lemma~\ref{ret1}
and Lemma~\ref{bias}, we have that $E_{Q_y^\omega} ( \ell_y(z)) \leq
e^{-\delta \kappa y^{1/4}/ 2} 2( R+1)y$ for each $\omega \in
A_{\delta,y}^c$ and each $z,\ -Ry\le z \leq -y^{1/4}$.
\end{proof}
  
Finally, we shall deal with the last term in (\ref{tf}).  
\begin{lemma} \label{bd2}
  $\displaystyle\lim_{y\to\infty}\frac{1}{y}\,E_{Q_y} \Big( \sum_ {-y^{1/4}
      \leq z < X_y} \ell_y(z)\Big)=0$.
\end{lemma}

\begin{proof}
  Set $T=\sum_ {-y^{1/4} \leq z < X_y} \ell_y(z)$.  Then for each
  $x \geq y^{3/4}$,
  \begin{equation*}
    Q_y(T \geq x)\leq Q_y\Big( \sum_{n=0}^{\tau_{x^{1/3}}} \I_{\{S_n
      \geq -y^{1/4}\}} \geq x\Big)+ Q_y(X_y>x^{1/3}).
  \end{equation*}
  The first probability is dominated by the corresponding probability
  for an unconditioned simple random walk by Lemma~\ref{bias} and so
  is bounded by $Ce^{-c x^{1/3}}$ for suitable nontrivial $c$ and $C$, while
  by (\ref{expbd2}) and Lemma~\ref{geomb}, the second term is
  similarly bounded.
\end{proof}
Lemmas~\ref{bd1} and \ref{bd2} imply Lemma~\ref{taux}, and we are done.
\end{proof} 

\appendix
\section{}

\begin{proof}[Proof of Lemma~\ref{conw}]
  We shall give a proof of the first statement. The second one is even
  simpler. Recall that $E^0$ denotes the expectation with respect to
  the simple symmetric random walk measure $P^0$. Since
  $P^0(\tau_r<\infty)=1$, we shall drop $\I_{\{\tau_r<\infty\}}$ when
  appropriate. For each $C>0$
  \begin{align*}
    E_{\Qb_{0,r}}(\tau_r)&=E_{\Qb_{0,r}}(\tau_r;\tau_r\le
    Cr^3)+E_{\Qb_{0,r}}(\tau_r;\tau_r> Cr^3)\\&\le Cr^3+
    \frac{E^0\left(\tau_r\I_{\{\tau_r>Cr^3,\tau_r<\tau_0^{(2)}\}}\right)} {\IE
      E^0\left(e^{-\sum_{n=0}^{\tau_r-1}V(S_n,\omega)}\I_{\{\tau_0^{(2)}>\tau_r\}}
        \prod_{x=1}^{r-1}\I_{\{\ell^r(x)\ge
          2\}}\right)}\\&\le Cr^3+\frac{E^0(\tau_r
        \I_{\{\tau_r>Cr^3,\tau_r<\tau_{-r}\}})}
      {2^{-2r-1}e^{-\Lambda(2)r}e^{-\Lambda(3)+\Lambda(2)}}.
  \end{align*}
  In the transition from the second to the third line, the lower bound
  on the denominator was obtained by choosing a particular path, which
  visits every $x$, $0<x<r$, exactly twice before hitting $r$ except in the case where $r$ is even when $r-2$ is visited three times. 
The numerator in the third line is equal to
\begin{multline*}
  E^0\left[(\tau_r\wedge \tau_{-r}) \I_{\{\tau_r>Cr^3,\tau_r<\tau_{-r}\}}\right]
  \le \\\left(E^0[(\tau_r\wedge
    \tau_{-r})^2]P^0(\tau_r\wedge \tau_{-r}
    >Cr^3)\right)^{1/2}\le C_0r^2e^{-C'r}.
\end{multline*}
In the last line we used two basic facts about a simple symmetric
random walk: (a) $E^0[(\tau_r\wedge \tau_{-r})^2]\le C_0r^4$ and (b)
the probability that the exit time from the strip $(-r,r)$ exceeds
$Cr^3$ is bounded by $e^{-2C' r}$, where $C'\to\infty$ as $C\to\infty$
(this follows from the invariance principle and a compactness
argument). Choosing large enough $C$ we can ensure that $C'>2\log
2+\Lambda(2)$.
\end{proof}
\begin{proof}[Proof of Lemma~\ref{ind}]
  We shall prove (\ref{ind1}). The proof of (\ref{ind2}) is the same.
  Let $\cal{E}=\cal{E}_{x-}\cap \{V(x)\in[\kappa,K]\}\cap
  \cal{E}_{x+}$, then
\[Q_{0,y}(\cal{E})=\frac{\IE E^0\left(e^{-\sum_{n=0}^{\tau_y-1}
      V(S_n)}\I_{\{\tau_0^{(2)}>\tau_y\}}\I_{\cal{E}}\right)}{\IE
  E^0\left(e^{-\sum_{n=0}^{\tau_y-1}
      V(S_n)}\I_{\{\tau_0^{(2)}>\tau_y\} }\right) } =:\frac{I}{II}.\] We
start with the numerator. Decomposing the path space according to the
number of visits to $x$ and applying the strong Markov property we get
\begin{align*}
  I&=\sum_{m=1}^\infty \IE E^0\left(e^{-\sum_{n=0}^{\tau_y-1}
      V(S_n)}\I_{\{\tau_0^{(2)}>\tau_y\}}\I_{\{\ell_y(x)=m\}}\I_{\cal{E}}\right)\\
  &= \sum_{m=1}^\infty \IE E^0\left(e^{-\sum_{n=0}^{\tau_x^{(m)}-1}
      V(S_n)}\I_{\{\tau_0^{(2)}>\tau_x^{(m)}\}}E^x\left(e^{-\sum_{n=0}^{\tau_y-1}
        V(S_n)}\I_{\{\tau_x^{(2)}>\tau_y\}}
    \right)\I_{\cal{E}}\right)\\ &\le \sum_{m=1}^\infty
  e^{-(m-1)\kappa}\IE E^0\left(e^{-\sum_{n=0}^{\tau_x-1}
      V(S_n)}\I_{\{\tau_0^{(2)}>\tau_x\}}E^x\left(e^{-\sum_{n=0}^{\tau_y-1}
        V(S_n)}\I_{\{\tau_x^{(2)}>\tau_y\}}
    \right)\I_{\cal{E}}\right)\\ &= \frac {1}{1-e^{-\kappa}}\,\IE
  E^0\left(e^{-\sum_{n=0}^{\tau_x-1}
      V(S_n)}\I_{\{\tau_0^{(2)}>\tau_x\}}\I_{\cal{E}_{x-}}\right) \IE
  E^x\left(e^{-\sum_{n=0}^{\tau_y-1}
      V(S_n)}\I_{\{\tau_x^{(2)}>\tau_y\}}\I_{\cal{E}_{x+}} \right)
\end{align*}
We also have that
\begin{multline*}
  II\ge \IE
  E^0\left(e^{-\sum_{n=0}^{\tau_x-1}V(S_n)-\sum_{n=\tau_x}^{\tau_y-1}V(S_n)}
    \I_{\{\tau_0^{(2)}>\tau_x\}}\I_{\{\tau_x^{(2)}>\tau_y\}}\right)\\
  = \IE E^0\left(e^{-\sum_{n=0}^{\tau_x-1}V(S_n)}
    \I_{\{\tau_0^{(2)}>\tau_x\}}\right)\IE
  E^x\left(e^{-\sum_{n=0}^{\tau_y-1}V(S_n)}\I_{\{\tau_x^{(2)}>\tau_y\}}\right).
\end{multline*}
Taking the ratio of the above two estimates completes the proof.
\end{proof}
\begin{proof}
  [Proof of Lemma~\ref{Zzz}] 
  Consider a terminating renewal sequence $(L_i)_{i\ge 1}$ defined as
  follows: \[L_0=1,\quad L_{i+1}=1+\max_{\tau_{L_i}\le
    n\le\tau_{L_i}^{(2)}}Y_n,\ i\ge 0.\] The kernel of this sequence,
  $q(r)$, has a non-trivial mass at infinity, $q(\infty)=2p-1$, which
  is equal to the probability of a random walk to never return to its
  current location. Moreover, it is elementary to compute that
  $q(r)<e^{-c_0 r}$, $r\in\IN$, for some positive $c_0$. Let
  $N=\min\{i\ge 0:\ L_{i+1}=\infty\}$. Then $N$ has a geometric
  distribution with parameter $2p-1$ and for $m>1$
  \begin{multline*}
    P(B(0,m))=\sum_{j=1}^\infty
  P(L_j\ge m\,|\,N=j)P(N=j)\\\le e^{-\epsilon (m-1)}\sum_{j=1}^\infty
  E(e^{\epsilon (L_j-L_0)}|N=j)P(N=j)\le e^{-\epsilon (m-1)}\sum_{j=1}^\infty
  (K(\epsilon))^j P(N=j),
  \end{multline*}
  where $K(\epsilon):=E(e^{\epsilon (L_1-L_0)}|L_1<\infty)\to 1$ as
    $\epsilon\to 0$. This implies the statement of the lemma.
\end{proof}  
\begin{lemma}
  \label{ret}
  Let $z\in\IZ$ and $\omega\in \{V(z,\omega)\ge
  \kappa\}\cap\{Z^\omega_y>0\}$. Then \[Q^\omega_y(\ell_y(z)>1)\le
  e^{-\kappa}.\]
\end{lemma}
\begin{proof}
  This is an obvious statement, which says that to return to $z$ it is
  necessary not to get absorbed when leaving $z$ after the first
  visit. Formally, using the strong Markov property of the killed
  random walk we get
  \begin{align*}
    &Q^\omega_y(\ell_y(z)>1)=
    \check{P}^\omega_0(\ell_y(z)>1\,|\,\tau_y<\infty) \le\\ &
    \check{P}^\omega_0(\tau_z^{(2)}<\tau_y\,|\,\tau_y<\infty,\tau_z<\infty)
    =\check{P}^\omega_z(\tau_z^{(2)}<\tau_y\,|\,\tau_y<\infty)= \\&
    \frac{\check{P}^\omega_z(\tau_y<\infty\,|\,\tau^{(2)}_z<\tau_y)}
    {\check{P}^\omega_z(\tau_y<\infty)}\,
    \check{P}^\omega_z(\tau^{(2)}_z<\tau_y)
     = \check{P}^\omega_z(\tau_z^{(2)}<\tau_y)\le
    e^{-\kappa}.\qedhere
  \end{align*}
\end{proof}
\begin{proof}[Proof of Lemma~\ref{ret1}]
  Part (a) follows from Lemma~\ref{ret}. For (b), by the Markov
  property we have
  \begin{align*}
    Q^{\omega,x_{i+r}}_y(\tau_{x_i}<\tau_y)&=\check{P}^{\omega,x_{i+r}}
    (\tau_{x_i}<\tau_y)\, \frac{\check{P}^{\omega,x_{i+r}}
      (\tau_y<\infty\,|\,\tau_{x_i}<\tau_y)
    }{\check{P}^{\omega,x_{i+r}} (\tau_y<\infty)}\\&\le e^{-\kappa r}
    \frac{\check{P}^{\omega,x_i}
      (\tau_y<\infty)}{\check{P}^{\omega,x_{i+r}} (\tau_y<\infty)}\le
    e^{-\kappa r}.\qedhere
  \end{align*}
\end{proof}

\begin{proof}[Proof of Lemma~\ref{0}]
      Since $Z_{0,y}\le Z_y$, we only need to show that 
      \[\liminf_{y\to\infty}\dfrac{\log Z_{0,y}}{y} \ge
      -\beta.\] By comparison with a simple symmetric random walk (see
      Lemma~\ref{bias}), we have that for $\IP$-a.e.\
      $\omega\in\{Z^\omega_y>0\}$ \[Q_y^\omega(\tau_y<\tau_0^{(2)})\ge
      P^0(\tau_y<\tau_0^{(2)})=\frac{1}{2y}.\]
      Therefore,
      \begin{equation*}
        Z_{0,y}=Q_y(\tau_y<\tau_0^{(2)})Z_y= 
        \IE(Q_y^\omega(\tau_y<\tau_0^{(2)})Z_y^\omega)\ge
        \frac{Z_y}{2y}.
      \end{equation*}
This finishes the proof.
    \end{proof}
      \begin{proof}[Proof of Lemma~\ref{bias}]
        This is a consequence of the strong Markov property of the
        killed random walk.
        \begin{multline*}
          Q_y^\omega(\ell_x(z)>m)=\\(Z^\omega_y)^{-1}
            E^0\left(e^{-\sum_{n=0}^{\tau_y-1}
                V(S_n,\omega)}\,\big|\,\tau_z ^{(m+1)}<\tau_x\right)
            P^0(\tau_z^{(m+1)}<\tau_x)\le \\(Z^\omega_y)^{-1}E^0\Big(
            e^{-\sum_{n=0}^{\tau_z-1} V(S_n,\omega)}
            E^0\Big(e^{-\sum_{n=\tau_z^{(m+1)}}^{\tau_y-1}
              V(S_n,\omega)}\,\big|\,\tau_z
            ^{(m+1)}<\tau_x\Big)\Big)\\\times P^0(\ell_x(z)>m)=P^0(\ell_x(z)>m).
        \qedhere
\end{multline*}
\end{proof}
       
The following follows from basic properties of simple random walks.
      \begin{lemma} \label{a2}
        Let $0\le x_1<x_2\le y$ and $\zeta=\max\{n:\,0\le n\le
        \tau_y,\,S_n=x_1\}$. Then
        \begin{equation}
          \label{u}
          E^0\left(e^{-\sum_{n=\zeta}^{\tau_{x_2}-1}V(S_n,\omega)}\right)= 
          E^{x_1}\left(e^{-\sum_{n=0}^{\tau_{x_2}-1}V(S_n,\omega)}\,\Big|\, 
            \tau_{x_1}^{(2)}>\tau_{x_2}\right).
        \end{equation}
      \end{lemma}
\begin{proof}[Proof of Lemma~\ref{lemA}]
  Let $\zeta=\max\{n:\,0\le n\le \tau_y,\ S_n=0\}$. Then
\begin{align*}
  Q_y(B)&=\frac{\IE\left(E
      (e^{-\sum_{n=0}^{\tau_y-1}V(S_n,\omega)})\I_B\right)}
  {\IE\left(E e^{-\sum_{n=0}^{\tau_y-1}V(S_n,\omega)}\right)}\le
  \frac{\IE\left(E
      (e^{-\sum_{n=\zeta}^{\tau_y-1}V(S_n,\omega)})\I_B\right)}
  {\IE\left(E
      (e^{-\sum_{n=0}^{\tau_y-1}V(S_n,\omega)}
 \I_{\{\tau_0^{(2)}>\tau_y\}})\right)}\\
  &=\frac{\IE\left(
      Ee^{-\sum_{n=\zeta}^{\tau_y-1}V(S_n,\omega)}\right)\IP(B)}
  {\IE\left(E
      (e^{-\sum_{n=0}^{\tau_y-1}V(S_n,\omega)}\,|\,\tau_0^{(2)}>\tau_y)\right)
    P^0(\tau_0^{(2)}>\tau_y)}\overset{(\ref{u})}{=}2y\IP(B).\qedhere
\end{align*}
\end{proof}
Finally, we turn to the expected (with respect to $Q_y$) total time
spent by the random walk below $-R$. Whenever the event $A$ depends
only on the random walk, we can rewrite $Q_y(A)$ as follows:
\begin{equation}
  \label{Q}
  Q_y(A)=\frac{E^0
    \left(\I_Ae^{-\sum_{x<y} \Lambda_V(\ell_y(x))}\right)}{
    E^0\left(e^{-\sum_{x<y} \Lambda_V(\ell_y(x))}\right)},\ 
  \text{where}\ \Lambda_V(t): 
  =-\log \mathbb{E}e^{-t V(0)}.
\end{equation}

\begin{lemma}\label{vn}
  For each $R>\beta_V(1)/\Lambda_V(1)-1$
  there is a $c>0$ such that for all sufficiently large $y$
  \begin{equation}
    \label{vnb}
    E_{Q_y} \Big( \sum_{z<-Ry} \ell_y(z) \Big) \le
  e^{-cy}.
  \end{equation}
\end{lemma}
\begin{proof}
  Let $R>\beta_V(1)/\Lambda_V(1)-1$. Choose $\epsilon>0$ so that
  $\Lambda_V(1)(R+1)>\beta_V(1)+\epsilon$. Since $Z_y\ge
  e^{-(\beta_V(1)+\epsilon/2)y}$ for all sufficiently large $y$, we
  get
\begin{align*}
  E_{Q_y} \Big(  &\sum_{z<-Ry} \ell_y(z) \Big)=\sum_{z<-Ry}\sum_{m=1}^\infty Q_y(\ell_y(z)\ge m)  \\
  &=\frac{1}{Z_y}\,\sum_{z<-Ry}\sum_{m=1}^\infty
  E^0\left(\I_{\{\ell_y(z)\ge m\}}
    e^{-\sum_{x<y}\Lambda_V(\ell_y(x))}\right)\\&\le
  e^{(\beta_V(1)+\epsilon/2)y} \sum_{z<-Ry}\sum_{m=1}^\infty
  E^0\left(\I_{\{\ell_y(z)\ge m\}}
    e^{-\sum_{x<y}\Lambda_V(\ell_y(x))}\right)\\&\le
  e^{(\beta_V(1)+\epsilon/2)y} \sum_{z<-Ry}\sum_{m=1}^\infty
  E^0\left(\I_{\{\ell_y(z)\ge m\}}
    e^{-\Lambda_V(m)-\sum_{\stackrel{x<y}{x\ne
          z}}\Lambda_V(\ell_y(x))}\right)\\&\le
  e^{(\beta_V(1)+\epsilon/2)y}
  \sum_{z<-Ry}e^{-\Lambda_V(1)(|z|+y-1)}\sum_{m=1}^\infty
  P^0(\ell_y(z)\ge m) e^{-\Lambda_V(m)}.
\end{align*}
Since for $z<0$  
\[P^0(\ell_y(z)\ge m)=
\left(1-\frac{1}{2(y+|z|)}\right)^{m-1}\frac{y}{y+|z|},\]
we obtain 
\begin{align*}
&\sum_{m=1}^\infty P^0(\ell_y(z)\ge m) e^{-\Lambda_V(m)}=
\frac{y}{y+|z|}\sum_{m=1}^\infty \mathbb{E}\,
e^{-mV}\left(1-\frac{1}{2(y+|z|)}\right)^{m-1}\\&=\frac{y}{y+|z|}\,
\mathbb{E}\left(\frac{e^{-V}}{1-e^{-V}\left(1-\frac{1}{2(y+|z|)}\right)}\right)=
\mathbb{E}\left(\frac{2y}{2(e^V-1)(y+|z|)+1}\right)\le 2y.
\end{align*}
Therefore,
\begin{align*}
 E_{Q_y} \Big( \sum_{z<-Ry} \ell_y(z) \Big) &\le 2ye^{(\beta_V(1)+\epsilon/2)y}
  \sum_{z<-Ry}e^{-\Lambda_V(1)(|z|+y-1)}  \\&=2ye^{(\beta_V(1)+\epsilon/2-
    \Lambda_V(1)(R+1))y} \sum_{x=0}^\infty
  e^{-\Lambda_V(1)x}\\&=2y\,e^{(\beta_V(1)+\epsilon/2-\Lambda_V(1)(R+1))y}\,\frac{1}
  {1-e^{-\Lambda_V(1)}}\\&\le
  e^{-(\Lambda_V(1)(R+1)-\beta_V(1)-\epsilon)y}
\end{align*}
for all sufficiently large $y$.
  \end{proof}

\subsection*{Acknowledgements}

This paper has benefited from a self-evidently close reading from the
referee and has been improved as a result. In particular, we thank him
or her for drawing our attention to the paper \cite{GdH92} and for
streamlining the proof of Lemma~\ref{probexp}.

E.\ Kosygina was partially supported by the PSC-CUNY award \#
63393-00-41.  T.\ Mountford was partially supported by the Swiss NSF
grant \# 200021-129555.

\bibliographystyle{amsalpha}

\begin{thebibliography}{99}
\bibitem[AZ96]{AZ96} {\sc S.\ Albeverio; X.\  Y.\ Zhou} (1996) Free energy
and some sample path properties of a random walk with random
potential.  \textit{J. Statist. Phys.} \textbf{83}, no.\ 3-4, 573--622.
\bibitem[DZ98]{DZ98}{\sc A.\ Dembo; O.\ Zeitouni} (1998) Large
  Deviations Techniques and Applications.  \textit{Springer,
    Berlin}, 396 pp.
\bibitem[Do01]{Do01}{\sc J.\ L.\ Doob} (2001) Classical Potential
  Theory and its Probabilistic Counterpart.  \textit{Springer,
    Berlin}, 846 pp.
\bibitem[Du05]{Du05}{\sc R.\ Durrett} (2005) Probability: Theory and Examples.
  \textit{Duxbury, Inc., Belmont}, 496 pp.
\bibitem[Fe68]{Fe68} {\sc W.\ Feller} (1968) An introduction to
  probability theory and its applications. Vol.\ I.  Third edition
  \textit{John Wiley \& Sons, Inc., New York-London-Sydney}, xviii+509 pp.
\bibitem[Fe71]{Fe71} {\sc W.\ Feller} (1971) An introduction to
  probability theory and its applications. Vol.\ II.  Second edition
  \textit{John Wiley \& Sons, Inc., New York-London-Sydney}, xxiv+669
  pp.
\bibitem[Fl07]{Fl07} {\sc M.\ Flury} (2007) Large deviations and phase
  transition for random walks in random nonnegative potentials.
  \textit{Stochastic Process. Appl.}  \textbf{117}, no.\ 5, 596--612.
\bibitem[Fl08]{Fl08} {\sc M.\ Flury} (2008) Coincidence of Lyapunov
  exponents for random walks in weak random potentials.  \textit{Ann. Probab.}
  \textbf{36}, no.\ 4, 1528--1583.
\bibitem[GdH92]{GdH92} {\sc A. Greven; F. den Hollander} Branching
  random walk in random environment: phase transitions for local and
  global rates.  \textit{Probab. Theory Related Fields} \textbf{91},
  no.\ 2, 195-249.
\bibitem[Gr99]{Gr99}{\sc G.\ Grimmett} (1999) Percolation.  Second
  edition.  \textit{Grundlehren der Mathematischen Wissenschaften},
  \textbf{321}, Springer-Verlag, Berlin, xiv+444 pp.
\bibitem[IV10a]{IV10a} {\sc D.\ Ioffe; Y.\ Velenik} (2010) Crossing random
  walks and stretched polymers at weak
  disorder. \textit{Ann. Probab.} (to appear).
\bibitem[IV10b]{IV10b} {\sc D.\ Ioffe; Y.\ Velenik} (2010) Stretched
  polymers in random environment. \textit{arXiv:1011.0266v1
    [math.PR]}, 28 pp.
\bibitem[IV11]{IV11} {\sc D.\ Ioffe; Y.\ Velenik} (2011) Self-attracting random walks: the case of cretical drifts. \textit{arXiv:1104.4615v2
    [math.PR]}, 24 pp.
\bibitem[Kh96]{Kh96} {\sc K.\ M.\ Khanin} (1996) Random walks in a random
potential: loop condensation effects.  \textit{Internat. J. Modern
  Phys. B} \textbf{10}, no.\ 18-19, 2393--2404. 
\bibitem[KMZ10]{KMZ10} {\sc E.\ Kosygina; T.\ Mountford;\ M.\
    Zerner} (2010) Lyapunov exponents of Green's functions for random
  potentials tending to zero. \textit{Probab. Theory Related Fields},
  Online First, 17 pp.
\bibitem[Po97]{Po97} {\sc T.\ Povel} (1997) Critical large deviations of
  one-dimensional annealed Brownian motion in a Poissonian potential.
  \textit{Ann. Probab.}, \textbf{25}, no.\ 4, 1735--1773.
\bibitem[Si95]{Si95} {\sc Y.\,G.\ Sinai} (1995) A remark concerning random
  walks with random potentials.  \textit{Fund. Math.}  \textbf{147},
  no.\ 2, 173--180.
\bibitem[Sz95]{Sz95}{\sc A.-S.\ Sznitman} (1995).  Crossing velocities and
  random lattice animals.  \textit{Ann. Probab.} \textbf{23}, no.\ 3,
  1006--1023.
\bibitem[Sz98]{Sz98} {\sc A.-S.\ Sznitman} (1998) Brownian motion,
  obstacles and random media.  \textit{Springer Monographs in
    Mathematics.}  Springer-Verlag, Berlin. xvi+353 pp.
\bibitem[Sz00]{Sz00} {\sc A.-S.\ Sznitman} (2000) Slowdown estimates and
  central limit theorem for random walks in random environment.
  \textit{J. Eur. Math. Soc. (JEMS)} \textbf{2}, no.\ 2, 93--143.
\bibitem[Wu98]{Wu98}{\sc M.\,V.\ W\"uthrich} (1998) Superdiffusive behavior
  of two-dimensional Brownian motion in a Poissonian potential.
  \textit{Ann. Probab.}  \textbf{26}, no.\ 3, 1000--1015.
\bibitem[Ze98]{Ze98}{\sc M.\,P.\,W.\ Zerner} (1998) Directional decay
  of the Green's function for a random nonnegative potential on ${\bf
    Z}^d$. \textit{Ann. Appl. Probab.}, \textbf{8}, no.\ 1, 246--280.
\bibitem[Ze00]{Ze00}{\sc M.\,P.\,W.\ Zerner} (2000) Velocity and
  Lyapounov exponents of some random walks in random environment.
  \textit{Ann. Inst. H. Poincar\'e Probab. Statist.}  \textbf{36},
  no.\ 6, 737--748.
\bibitem[Zy09]{Zy09}{\sc N.\ Zygouras} (2009) Lyapounov norms for
  random walks in low disorder and dimension greater than three.
  \textit{Probab. Theory Related Fields} \textbf{143}, no.\ 3-4,
  615--642.

\end{thebibliography}

{\sc \small
\begin{tabular}{ll}
Department of Mathematics& \hspace*{20mm}\'Ecole Polytechnique F\'ed\'erale\\
Baruch College, Box B6-230& \hspace*{20mm}de Lausanne\\
One Bernard Baruch Way&\hspace*{20mm}D\'epartement de math\'ematiques\\
New York, NY 10010, USA&\hspace*{20mm}1015 Lausanne, Switzerland\\
{\verb+elena.kosygina@baruch.cuny.edu+}& \hspace*{20mm}{\verb+thomas.mountford@epfl.ch+}
\end{tabular}\vspace*{2mm}
\end{document}